\documentclass{siamltex}

\usepackage{amssymb,graphicx,color,rotating,amsmath,amsfonts}
\usepackage{hyperref} 
\usepackage{graphicx}
\newcommand{\RR}{{\mathbb{R}}}

\newcommand{\pc}{{\cal P}^{IPF}_k}
\newcommand{\pd}{{\cal P}^{BDF}_k}
\newcommand{\ascp}{{\sc as-gmres-ipf} }
\newcommand{\asbd}{{\sc as-minres-bdf} }
\newcommand{\asbt}{{\sc as-bpcg-bt} }

\newtheorem{example}[theorem]{Example}
\newtheorem{remark}[theorem]{Remark}

\usepackage{slashbox}
\usepackage{booktabs}
\usepackage[table]{xcolor}

\begin{document}     
\bibliographystyle{siam}

\pagestyle{myheadings}
\markboth{M. Porcelli, V. Simoncini and M. Tani}{Preconditioning for constrained control problems}

\title {Preconditioning of Active-Set Newton Methods for 
PDE-constrained Optimal Control Problems%
\thanks{Version of {\color{black}May 22}, 2015. This work was 
 partially supported by {\em National Group of Computing Science (GNCS-INDAM)}}}

\author{Margherita Porcelli$^\dagger$
\and Valeria Simoncini$^\dagger$
\and  Mattia Tani\thanks{%
Universit\`a di Bologna, Dipartimento di Matematica, 
Piazza di Porta S. Donato 5, 40127  Bologna, Italy. Emails: 
{\tt  \{margherita.porcelli,valeria.simoncini,mattia.tani2\}@unibo.it}}}
%
\maketitle
\begin{abstract} 
We address the problem of preconditioning a sequence of
saddle point linear systems arising in the solution
of PDE-constrained optimal control problems via active-set Newton methods, with
control and (regularized) state constraints.
We present two new preconditioners based on a full block matrix factorization of the
Schur complement of the Jacobian matrices, where the active-set blocks are merged into
the constraint blocks. 
We discuss the robustness of the new preconditioners with respect to the
parameters of the continuous and discrete problems.
Numerical experiments on 3D problems are presented, including 
comparisons with existing approaches based on preconditioned conjugate
gradients in a nonstandard inner product. 
\end{abstract}



\section{The problem} 
In this work we consider the family of PDE-constrained optimization problems of the form 
\begin{equation}
\begin{array}{l}\label{pb_gen}
\displaystyle\min_{{\rm y,u}}  \frac 1 2 \|{\rm y-y_d}\|_{L^2(\Omega)}^2 + \frac{\nu}{2} \|{\rm u}\|_{L^2(\Omega)}^2 \\
 \\
\mbox{ s.t. }  
\left \{ \begin{array}{ll}{\rm -\Delta y - {\beta}\cdot \nabla y  = u} & \mbox{ in }  \Omega \\
	{\rm  y = \bar y} & \mbox{ on }  \partial\Omega \\
	{\rm a }\le  \alpha_u {\rm u} + \alpha_y {\rm y  \le b} & \mbox{ a.e. in } \Omega  ,
\end{array}\right .  
\end{array}\end{equation}
where $\nu\in \RR^+$ is a regularization parameter, ${\rm y_d}$ is a given function representing the desired state, and $\Omega$ is a
domain in $\RR^d$ with $d=2,3$. The state $\rm y$ and the control $\rm u$ are linked via an
elliptic convection-diffusion equation with convection direction $\beta \in \RR^d$. Dirichlet boundary conditions are assumed.
Moreover, we assume the presence of box constraints of the form ${\rm  a \le  \alpha_u u + \alpha_y y  \le b}$ a.e. in $\Omega$,
where we assume ${\rm a(x)<b(x) }$ a.e. in $\Omega$ and  $\alpha_u,\alpha_y$ nonnegative scalars such that $\max\{\alpha_y,\alpha_u\}>0$.
By varying the parameters $\alpha_u,\alpha_y$, we obtain optimal control problems 
with different inequality constraints. In particular, we will consider three specific choices 
which yield well studied problems. 
The first is  $(\alpha_u,\alpha_y) = (1,0)$, that is 
\begin{equation}\label{pb_cc}
{\rm a \le u \le b} \quad  \mbox{ a.e. in } \Omega, 
\end{equation}
which is refereed to as the optimal control problem with {\em Control Constraints} (CC).
The second one is $(\alpha_u,\alpha_y )= (\epsilon,1)$ yielding 
 an optimal control problem with {\em Mixed Constraints} (MC) of the form
\begin{equation}\label{pb_mc}
{\rm a \le  \epsilon u + y  \le b}\quad  \mbox{ a.e. in } \Omega .
\end{equation}
The third choice is $(\alpha_u,\alpha_y )= (0,1)$ 
yielding optimal control problems with {\em State Constraints} (SC)
\begin{equation}\label{pb_sc}
{\rm a \le  y  \le b} \quad  \mbox{ a.e. in } \Omega .
\end{equation}
Mixed Constraints (\ref{pb_mc}) are commonly employed as a form of
regularization of the state-constrained problem,  where 
$\epsilon >0$ represents the regularization parameter \cite{Meyer2007}. Indeed,
pure state constrained problems are more complicated than control constrained ones, 
as in general the Lagrange multiplier associated with state constraints is only a 
measure, and therefore regularized versions with better regularity properties are 
needed {\color{black} to justify the employed solution methods in function
space.} 
{\color{black} 
Examples of typical employed regularizations are the one applied to the state 
constraints as in (\ref{pb_mc}) (see, e.g., \cite{Casas1986}) 
and the one that employs the Moreau-Yosida penalty function \cite{ItoKunish2003,StollPearsonWathen2011}.
We refer to \cite{Meyer2007}
for a discussion of this point, and for the use of primal-dual active
set strategies to deal with the regularized problem. I}n the following we shall see the
purely state-constrained problem as the limit case of the mixed-constrained one. As such, 
it may provide helpful information as a computational reference for the mixed-constrained problem
when $\epsilon$ is very small. 

We follow a {\em discretize-then-optimize} approach for the solution of problem (\ref{pb_gen})
as we first transform the original continuous problem into a standard Quadratic Programming (QP) problem
by a finite difference or finite element discretization, and then we numerically solve the first-order 
conditions of the fully discretized optimization problem. Issues related to the 
commutativity  between the discretize-then-optimize and the optimize-then-discretize approach 
for convection diffusion control equations have been addressed in \cite{PearsonWathen2013}.

Due to the presence of inequality constraints in the problem formulation, 
the optimality system is nonlinear. Moreover, its dimension will be very large as soon
as the desired accuracy requires a fine discretization of the partial differential
equation; the Lagrange multiplier approach also yields a structured (block) nonlinear
equation, thus further expanding the discrete problem size.  We therefore apply a Newton-type approach 
for the nonlinear equation solution,  and we use a Krylov subspace method to solve 
the arising sequence of large and sparse saddle point linear systems.
It is well-known that a computationally effective solution of the linear algebra phase
is crucial for the practical implementation of the Newton-Krylov method \cite{DApuzzo10} and it is widely
recognized that preconditioning is a critical ingredient of the iterative solver.

Existing preconditioners for constrained optimal control problems 
have been tailored for specific elements of
the family (\ref{pb_gen}) and are generally suitable for problems where the
operator characterizing the PDE is self-adjoint. Moreover, implementations based
on the preconditioned conjugate gradient method in a nonstandard inner
product have often been preferred, in spite of possible 
strong limitations \cite{HerzogSachs2010,StollPearsonWathen2011,StollWathen2012}.
The works \cite{StollPearsonWathen2011,StollWathen2012} are for CC problems governed by symmetric PDEs,
while \cite{HerzogSachs2010} considers problems with constraints (\ref{pb_cc})-(\ref{pb_sc}) but
mostly focuses on $\beta=0$. In particular, the problematic numerical behavior of the preconditioners proposed 
in \cite{HerzogSachs2010} for $\beta \neq 0$ motivated this work.

In this paper, we present two new preconditioners aimed at enhancing the solution of 
the linear algebra phase arising from the discretization of the family of optimal control 
problems involving state and/or control constraints of the form (\ref{pb_gen}).
We consider an indefinite preconditioner and a
symmetric and positive definite block diagonal preconditioner.
Both strategies rely on a general factorized approximation of the Schur complement,
and embed newly formed information of the nonlinear iteration, so that
they dynamically change as the nonlinear iteration proceeds. 
The proposed preconditioners are very versatile, as they allow to handle mixed constraints as well as 
the corresponding limit cases, that is control and state constraints. 
In particular, we derive optimality and robustness theoretical properties for the spectrum of the
preconditioned matrices, which hold for a relevant class of problem parameters; numerical
experiments support this optimality also in terms of CPU time.  
A broad range of numerical experiments on three test problems is reported, for a large selection  of the
four problem parameters ($\nu, \beta, \epsilon$ and the spatial mesh size $h$), 
indicating  only a mild sensitivity of the preconditioner with respect
to these values, especially when compared with existing approaches (for the parameters for
which these latter strategies are defined). In addition, in most cases the  indefinite preconditioner
outperforms by at least 50\% the block diagonal preconditioner, for the same Schur complement approximation.
%

The outline of the paper is as follows. Sections \ref{optpb} and \ref{semiNew} describe the
discrete problem and its formal numerical solution by an active-set Newton method. 
Section~\ref{sec:overview} reviews the preconditioning strategies that have been devised to solve (\ref{pb_gen}) for
some choices of the selected parameters. In Section \ref{sec:as_schur} a new general
approximation to the Schur complement is introduced and theoretically analyzed, while
its impact on the new global preconditioners is investigated in Section \ref{sec:prec}.
Section \ref{num_exp}  is devoted to a wide range of numerical results. In particular,
in Section \ref{sec:algo} we discuss some algorithmic details, while in Section \ref{sec:expes} we
report on our numerical experiments on three model problems.
Section \ref{conclusions} summarizes our conclusions.

The following notation will be used throughout the paper. For a given square matrix $A$, ${\rm spec}(A)$ denotes
the set of its eigenvalues.
The Euclidean norm for vectors and the induced norm for matrices is used;
$x^T$ denotes the transpose of the vector $x$.

\section{Description of the problem}

\subsection{The discrete optimization problem} \label{optpb}

Let $M$ represent the lumped mass matrices in an appropriate finite element space, and
$L$ be the discretization of the differential operator 
${\cal L}(\rm y) = -\Delta {\rm y} + \beta\cdot \nabla {\rm y}$; in particular,
$L$ is a nonsymmetric matrix of the form $L = K + C$, where $K$ is the symmetric and positive
definite discretization of the (negative)  Laplacian operator and $C$ is the  ``convection'' matrix. 
In the following we shall assume that $L+L^T \succeq 0$
\footnote{This requirement is satisfied when, for instance, upwind finite differences
over a regular grid, or upwind-type finite elements are used, with Dirichlet boundary conditions; see, e.g.,
\cite[Chapter 3]{Elman2005}, \cite{Golub1996}.}
Moreover, let $n_h$ be the dimension of the discretized space depending on the mesh size $h$ 
and let $ y,u, a, b\in\RR^{n_h}$ be the coefficients of
${\rm y,u, a, b}$ in the chosen finite element space basis.
Then, the discretization of problem (\ref{pb_gen}) is given by the following QP problem

\begin{equation}\label{dpb_gen}
\begin{array}{l}
\displaystyle \min_{y,u}  Q(u,y)=\frac 1 2 (y-y_d)^T M(y-y_d) + \frac{\nu}{2} u^TM u  \\
 \\
\mbox{ s.t. }  
\left \{ \begin{array}{l}
        Ly  = M u-d  \\
        a \le  \alpha_u u + \alpha_y y  \le b
\end{array}\right .  
\end{array}
\end{equation}
where $d$ represents the boundary data.
The Lagrangian function for problem (\ref{dpb_gen}) is given by
$$
{\mathcal{L}}(u,y,p,\mu) = Q(u,y) + (Ly-Mu+d)^Tp + (\alpha_u u + \alpha_y y -b)^T\mu_b +
(\alpha_u u + \alpha_y y -a)^T \mu_a
$$
where $p$ is the Lagrange multiplier associated with the linear equality constraint
and $\mu_a, \mu_b$ are the Lagrange multipliers associated with the lower and upper bound constraints.
The corresponding Karush-Kuhn-Tucker conditions are
\begin{equation}\label{kkt}
\begin{array}{l}
\nabla_y {\mathcal{L}} = M (y-y_d) + L^T p+ \alpha_y (\mu_b+\mu_a)  =0 \\
\nabla_u {\mathcal{L}} =\nu M u - M p + \alpha_u (\mu_b+\mu_a) =0 \\
L y  - M u +d = 0 \\
\mu_b \ge 0,\ \ \alpha_u u + \alpha_y y \le b, \ \ \mu_b^T(\alpha_u u + \alpha_y y -b) =0  \\
\mu_a \le 0,\ \ a \le \alpha_u u + \alpha_y y, \ \ \mu_a^T(a-\alpha_u u - \alpha_y y )=0  
\end{array}
\end{equation}
%
Setting $\mu = (\mu_b+\mu_a)$, the complementarity conditions in (\ref{kkt}) can be equivalently stated as the following nonlinear system 
$$C(u,y,\mu) =0$$
with $C$ the following complementary function
\begin{equation}\label{nlkkt}
C(u,y,\mu) = \mu - \max \{0, \mu + c(\alpha_u u + \alpha_y y - b)\} - \min \{0, \mu + c(\alpha_u u + \alpha_y y -a )\}, 
\end{equation}
with $c>0$.
Therefore,  the KKT system (\ref{kkt}) can be  reformulated as the following  nonlinear system
\begin{equation}\label{NL}
F(y,u,p,\mu) =\begin{bmatrix}
M (y-y_d) + L^T p  + \alpha_y \mu \\
\nu M u - M p + \alpha_u \mu  \\
L y - M u +d\\        
C(u,y,\mu)  \end{bmatrix}=0
\end{equation}
with $F:\RR^{4n_h}\rightarrow \RR^{4n_h}$, $y,u,p,\mu \in\RR^{n_h}$. 

\subsection{The active-set Newton method} \label{semiNew}
In the following we  recall a possible derivation of an active-set Newton type method for the solution
of the KKT nonlinear system (\ref{NL}) following the description made in \cite{Hintermuller2002}
where  nonsmooth analysis was used.

Let us define the sets of active and inactive  indices at the (discrete) optimal solution $(u^*,y^*)$
\begin{equation} \label{setAI}
{\mathcal{A}}_*  = {\mathcal{A}^b_*} \cup {\mathcal{A}_*^a}\ \mbox{ and }
\ {\mathcal{I}}_*  =  \{1,\dots,n_h\} \setminus  {\mathcal{A}_*},  
\end{equation}
where ${\mathcal{A}_*^b},{\mathcal{A}_*^a}$ are the set 
\begin{eqnarray*}
{\mathcal{A}_*^b} = \{ i \ | \ \mu_i^*+c (\alpha_u u^*_i + \alpha_y y^*_i-b_i)>0 \}, \quad\, 
{\mathcal{A}_*^a}  = \{ i \ | \  \mu_i^*+c(\alpha_u u^*_i + \alpha_y y^*_i - a_i)<0  \}.
\end{eqnarray*}
The  nonlinearity and nonsmoothness of the function $F$ in (\ref{NL}) are clearly gathered in the last block containing the complementarity function 
$C(u,y,\mu)$ defined in (\ref{nlkkt}).  
Hinterm\"uller et al. showed in \cite{Hintermuller2002} that the
functions $ v \rightarrow \min\{ 0,v\}$
and $ v \rightarrow \max\{ 0,v\}$
from $\RR^n \rightarrow \RR^n$ are slantly differentiable 
with slanting functions given by the diagonal matrices  $G_{\min}(v)$ and $G_{\max}(v)$
 with diagonal elements 
$$
G_{\min}(v)_{ii} = \left \{ \begin{array}{ll}
1 & \mbox{ if } v_i < 0 \\
0 & \mbox{ else }
\end{array}
\right ., \quad G_{\max}(v)_{ii} = \left \{ \begin{array}{ll}
1 & \mbox{ if } v_i > 0 \\
0 & \mbox{ else }
\end{array}
\right .$$
The choice of $G_{\min}$ and $G_{\max}$ suggests to use the following element $F'(y^*,u^*,p^*,\mu^*)\in \RR^{4n_h \times 4n_h}$ 
of the generalized Jacobian $\partial F(y^*,u^*,p^*,\mu^*)$ (\cite{Clarke1983})
\begin{equation}\label{jacgen}
F'(y^*,u^*,p^*,\mu^*) = \begin{bmatrix} 
M & 0  & L^T & \alpha_y I \\
0 & \nu M & -M & \alpha_u I  \\
L & - M & 0 & 0 \\        
c\alpha_y \Pi_{\mathcal{A}_*} & c \alpha_u  \Pi_{\mathcal{A}_*} & 0 & \Pi_{\mathcal{I}_*}   \end{bmatrix} ,
\end{equation}
to construct a ``semismooth'' Newton scheme. Here $\Pi_{\mathcal{C}}$ denotes a diagonal binary matrix with nonzero entries in $\mathcal{C}$,
and the sets $\mathcal{A}_*,\mathcal{I}_*$ are given in (\ref{setAI}).

Given the $k$th iterate $(y_k,u_k,p_k,\mu_k)$, let ${\mathcal{A}_k}$ and ${\mathcal{I}_k}$ be the current active and inactive sets 
where 
\begin{subequations} \label{AS-def} 
\begin{eqnarray}
{\mathcal{A}_k} & = & {\mathcal{A}^b_k} \cup {\mathcal{A}^a_k}, \quad 
{\mathcal{I}_k}  =  \{1,\dots,n_h\} \setminus  {\mathcal{A}_k}  \label{setIk}\\
{\mathcal{A}^b_k} & = & \{ i \ | \ ({\mu_k})_i + c(\alpha_u ({u_k})_i + \alpha_y ({y_k})_i -b_i)>0 \}\\
{\mathcal{A}^a_k} & = & \{ i \ | \  ({\mu_k})_i + c(\alpha_u ({u_k})_i + \alpha_y ({y_k})_i-a_i)<0 \} 
\end{eqnarray}
\end{subequations}
and let $n_{{\mathcal{A}}_k} = card({\mathcal{A}_k})$ be the current number of active constraints.
Using the Jacobian $F'$ in (\ref{jacgen}), the semismooth Newton iteration \cite{Hintermuller2002} applied to system (\ref{NL}) is the following:
$$
\begin{bmatrix}  
M & 0  & L^T & \alpha_y I \\
0 & \nu M & -M & \alpha_u I  \\
L & - M & 0 & 0 \\        
c \alpha_y \Pi_{\mathcal{A}_k} & c \alpha_u \Pi_{\mathcal{A}_k} & 0 & \Pi_{\mathcal{I}_k}
\end{bmatrix}
\begin{bmatrix} 
y_{k+1} \\
u_{k+1}  \\
p_{k+1} \\        
\mu_{k+1}
\end{bmatrix} =		
\begin{bmatrix} 
M y_d \\
0  \\
d \\        
c(\Pi_{\mathcal{A}_k^b} b+\Pi_{\mathcal{A}_k^a} a)
\end{bmatrix}.
$$
Setting ${(\mu_{k+1}})_{\mathcal{I}_k} = 0$ (the multiplier associated with the inactive
inequality constraints) and eliminating this variable, we obtain the sequence of Newton structured equations
\begin{equation}\label{new_eq}
J_k x_{k+1}  = f_k, \quad k = 1,2, \dots
\end{equation}
where $x_{k+1} = (y_{k+1},u_{k+1},p_{k+1},(\mu_{k+1})_{\mathcal{A}_k}) \in\RR^{3 n_h + n_{{\mathcal{A}}_k}}$, 
\begin{equation}\label{JK}
f_k = \begin{bmatrix} 
M y_d \\
0  \\
d \\        
P_{\mathcal{A}_k^b} b+P_{\mathcal{A}_k^a} a
\end{bmatrix}  ,\quad
  J_k =
\begin{bmatrix} 
M & 0  & L^T & \alpha_y P^T_{\mathcal{A}_k}\\
0 & \nu M & -M & \alpha_u P^T_{\mathcal{A}_k}  \\
L & - M & 0 & 0 \\        
\alpha_y P_{\mathcal{A}_k} & \alpha_u P_{\mathcal{A}_k} & 0 & 0
\end{bmatrix} , 
\end{equation}
where
$P_{\mathcal{C}}$ is a rectangular matrix consisting of those rows of $\Pi_{\mathcal{C}}$ which belong to the
indices in ${\mathcal{C}}$; with this notation, {$\Pi_{\cal C} = P_{\cal C}^T P_{\cal C}$}.
We remark that the value of $c$ has no influence on the solution of the Newton equation (\ref{new_eq})
but  affects the updating of the active sets ${\mathcal{A}_k}$ in (\ref{AS-def}).

The above semismooth Newton scheme was proved to be equivalent to 
the Primal-Dual active-set method for solving constrained optimal control problems in \cite{Hintermuller2002}
and this equivalence allowed to establish superlinear local and also global convergence results 
\cite{Hintermuller2002,Meyer2007,Kunisch2002}.
In fact, the active-set strategy works as a prediction technique in the sense that
it is proved that if $(u_k,y_k,p_k,\mu_k)\rightarrow (u^*,y^*,p^*,\mu^*)$, then there exists an index $\bar k$
such that ${\mathcal{A}_{\bar k}} = {\mathcal{A}_*}$ and  ${\mathcal{I}_{\bar k}} = {\mathcal{I}_*}$ \cite[Remark 3.4]{Hintermuller2002}.

Given $x_k$, the next  iterate $x_{k+1}$ is commonly computed by applying 
an iterative solver (in our case a preconditioned Krylov subspace method)
to the Newton equation (\ref{new_eq}),
and then generating a sequence of (inner) iterations $\{x_{k+1}^j\}_{j\ge0}$.
The inner iteration is started with $x^0_{k+1}=x_k$ and  stopped for $j_*>0$ such that  
\begin{equation}\label{crit}
\|J_{k}x_{k+1}^{j_*}-f_{k}\|\le \eta_k \|J_{k}x^0_{k+1}-f_{k}\| 
\end{equation}
and the next iterate $x_{k+1}$ is set equal to $x_{k+1}^{j_*}$. The 
scalar $\eta_k>0$ controls the accuracy in the solution of the unpreconditioned
linear system.
The choice $\eta_k=\eta_k^E$ with 
\begin{eqnarray}\label{ex_crit}
 \eta^{E}_k =  \tau_1, 
\end{eqnarray}
$k \ge 1$, with a small $\tau_1$ (e.g. $\tau_1=10^{-10}$)
allows us to compare various preconditioning techniques in solving
the linear system (\ref{new_eq}), while the nonlinear iteration remains substantially unaffected by the
use of each different inner strategy. This stopping criterion  was used in all 
our numerical experiments of Sections \ref{exp_comp} and \ref{exp_new}.

Occasionally, for some choice of problem parameters 
we have experienced slow convergence of the Newton method 
in the solution of CC problems. This prompted us to
also consider the adaptive choice $\eta_k=\eta_k^I$
\begin{eqnarray}\label{in_crit}
\eta^I_0=\tau_2,\quad \eta^I_{k}= \min \{\eta^I_{k-1}, \tau_3\|F(u_k,y_k,p_k,\mu_k)\|^2 \},
\end{eqnarray}
$k \ge 1$ (e.g. with $\tau_2 = 10^{-4}, \tau_3 = 10^{-2}$),
which gives rise to the ``inexact'' solution of the Newton system \cite{ew2,Kanzow2004,Porcelli13}. In particular,
(\ref{in_crit}) is intended
to give the desirably fast local convergence near a solution and, at the same time, to minimize
the occurrence of problem oversolving. 
We remark that the global convergence of the active set Newton method is no longer guaranteed 
if inexact steps are computed, but it is anyway expected for small values of the initial
forcing term $\eta^I_0$ \cite{Kanzow2004}.
Numerical tests with (\ref{in_crit}) are reported in Section \ref{exp_inex}.

The key step in the overall process is the efficient iterative solution of the
linear systems (\ref{new_eq}), for which preconditioning is mandatory.
The rest of the paper is thus devoted to the analysis of effective
preconditioning strategies. 

\section{Overview of the current approaches}\label{sec:overview}
In this section we review some of the preconditioning strategies that have been
explored in the literature for the solution 
of CC, MC and SC problems. In particular we consider the proposals \cite{HerzogSachs2010,StollWathen2012}
both based on the use of 
the Preconditioned Conjugate Gradient  method \cite{BramblePasciak1988,SchoberlZulehner2007} 
with a nonstandard inner product 
for the solution of the saddle point linear systems arising in
the active-set Newton method for solving (\ref{NL}).
This approach (from now on named {\sc bpcg})  was originally used in the context of saddle point linear systems 
for mixed approximations of elliptic problems 
by Bramble and Pasciak in \cite{BramblePasciak1988}, and then subsequently used
in different settings where similar linear systems arise; 
see, e.g.,  \cite{HerzogSachs2010,SchoberlZulehner2007} in our context. 
%
%
Herzog and Sachs in \cite{HerzogSachs2010} consider the solution of CC, MC and SC problems by
partitioning the Jacobian matrix $J_k$ as follows 
\begin{eqnarray}\label{eqn:J_part}
J_k = \left [\begin{array}{c c | c c}
M & 0  & L^T & \alpha_y P^T_{\mathcal{A}_k}\\
0 & \nu M & -M & \alpha_u P^T_{\mathcal{A}_k}  \\
\hline
L & - M & 0 & 0 \\        
\alpha_y P_{\mathcal{A}_k} & \alpha_u P_{\mathcal{A}_k} & 0 & 0
  \end{array}  \right ] =
\begin{bmatrix} A & B_k^T \\ B_k & 0 \end{bmatrix} ,
\end{eqnarray}
and therefore considering $(\alpha_y,\alpha_u)=(0,1)$ in the CC case, 
$(\alpha_y,\alpha_u)=(1,\epsilon)$ in the MC case and $(\alpha_y,\alpha_u)=(1,0)$ in the SC case. 
Following the approach presented in  \cite{SchoberlZulehner2007}, Herzog and Sachs 
proposed the preconditioner 
\begin{eqnarray}\label{prec_HS}
 {\cal P}^{\sigma,\tau}_k = 
\begin{bmatrix} I & 0 \\ B_k \widehat A(\sigma)^{-1} & I \end{bmatrix}
\begin{bmatrix} \widehat A(\sigma) & B_k^T \\ 0  & -\widehat S_k(\sigma,\tau) \end{bmatrix} ,
\end{eqnarray}
where $\widehat A(\sigma)$ and $\widehat S_k(\sigma,\tau)$
approximate the $(1,1)$ block $A$ and the Schur complement $S_k$ respectively, and are block 
diagonal matrices, while the scalars $\sigma$ and $\tau$ are suitably chosen positive scalars,
whose role will be made clear below. 
%
A feature of this approach is that 
the blocks of $ \widehat A(\sigma)$ and $\widehat S_k(\sigma,\tau)$ can be
 chosen as (approximations of) the inner product matrices of the spaces where the continuous 
unknowns $\rm (y,u)$ and $\rm (p,\mu)$ are sought.
%
In particular,
$$
\widehat A(\sigma) = \frac{1}{\sigma}\begin{bmatrix} K &  0 \\
                      0 & M
                     \end{bmatrix} \quad \mbox{ and }\quad  \widehat S_k(\sigma,\tau) =
                     \frac{\sigma}{\tau}\begin{bmatrix}
                                         K & 0 \\
                                         0 & P_{\mathcal{A}_k} M^{-1} P_{\mathcal{A}_k}^T 
                                        \end{bmatrix} ,
$$
where $K$ is associated with the scalar product in the discretized state space,
and the scaling parameters 
$\sigma$ and $\tau$ are associated with the bilinear forms underlying the considered problems
(note that $K=L$ if $\beta=0$). 

The role of the scalars $\sigma$ and $\tau$ is crucial
since they have to ensure that $\widehat{A}(\sigma) > A $ and $B_k \widehat{A}(\sigma)^{-1} B_k^T  > \widehat{S}_k(\sigma,\tau)$,
so that the preconditioned matrix
 $(\mathcal{P}_k^{\sigma,\tau})^{-1} J_k$ is positive definite with respect to the inner product defined by
$$ 
\mathcal{D}_k^{\sigma,\tau} = \mathcal{P}_k^{\sigma,\tau} - J_k = 
\begin{bmatrix} 
\widehat{A}(\sigma) - A & 0 \\ 0 & B_k \widehat{A}(\sigma)^{-1} B_k^T - \widehat{S}_k(\sigma,\tau) 
\end{bmatrix} .
$$
Under these conditions, the CG method in this non-Euclidean inner product can be used.

The spectral analysis provided in \cite[Corollary 2.3]{HerzogSachs2010}  for $L$ symmetric ($\beta=0$)
shows that the eigenvalues of $ (\mathcal{P}_k^{\sigma,\tau})^{-1} J_k $ are 
bounded independently of $h$, while they depend on $\nu$ in such a way that
the condition number of the preconditioned matrix is proportional to $1/\nu$; 
As a consequence, poor convergence of {\sc bpcg} for small values of $\nu$ is predicted, and
also verified experimentally.  Moreover,  the authors show that the (preconditioned) condition 
number in MC problems scales like $\epsilon^{-2}$ for small $\epsilon$,
making the use of the proposed preconditioner prohibitive for values
of $\epsilon$ smaller than $10^{-3}$. 
Regarding the analysis for problems with  $\beta=(\beta_1,0,0)$, $\beta_{1}>0$,
a deterioration of the convergence behavior for large values
of $\beta_{1}$ was theoretically analyzed for CC problems and confirmed in the few 
reported experiments. The difficulties in solving these problems are illustrated in the plots of 
Figure~\ref{herzog} which are in complete agreement with \cite[Figure 4]{HerzogSachs2010}, and
were obtained with the same codes\footnote{We thank Roland Herzog for providing us with all Matlab codes
used in \cite{HerzogSachs2010}.}, though on a different machine.
In particular, we emphasize the strong 
dependence on $\beta$ and $h$ of the preconditioned strategy.

 \begin{figure}[htb]
  \centering
      \includegraphics[width=2.5in,height=2.2in]{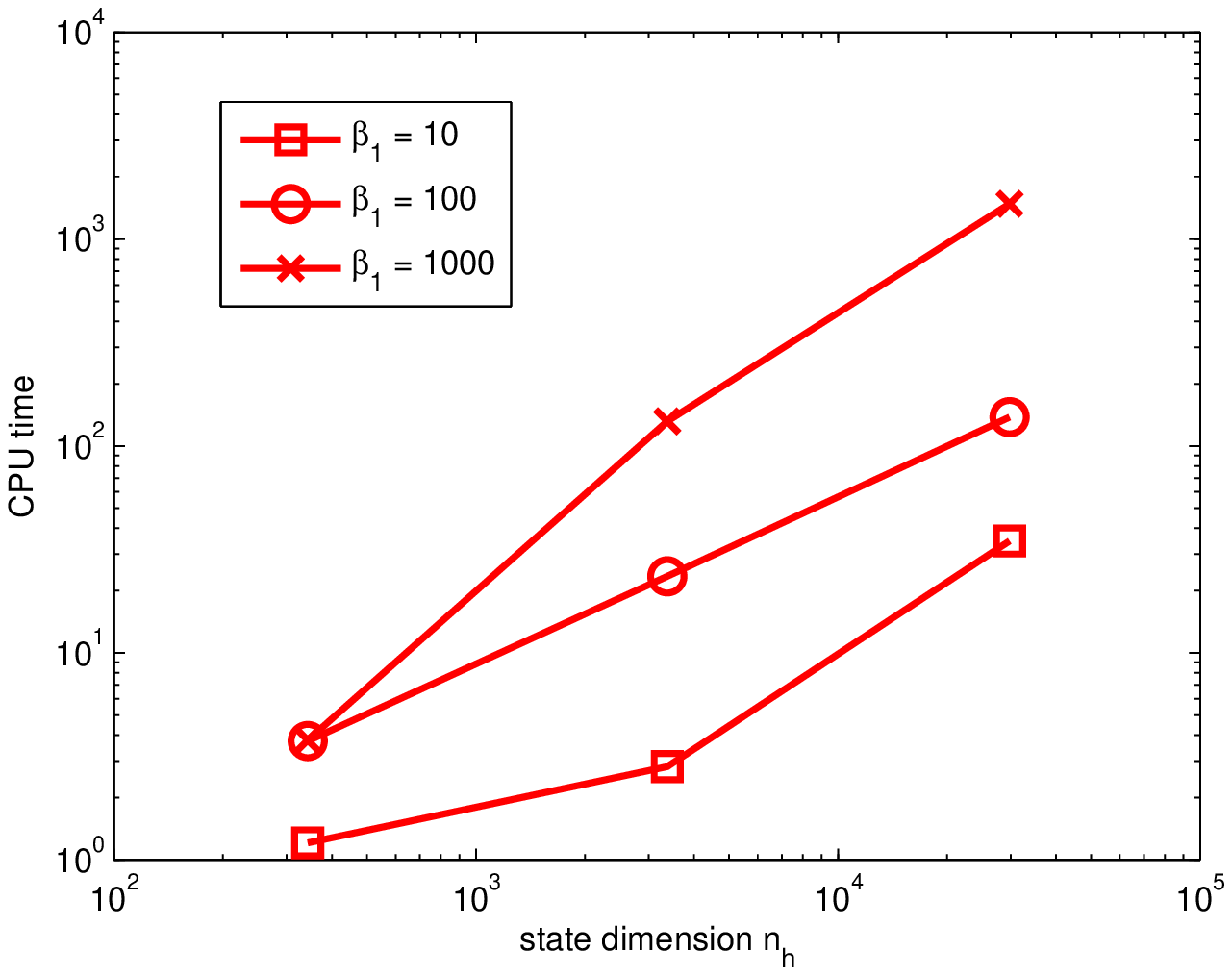}
    \includegraphics[width=2.5in,height=2.2in]{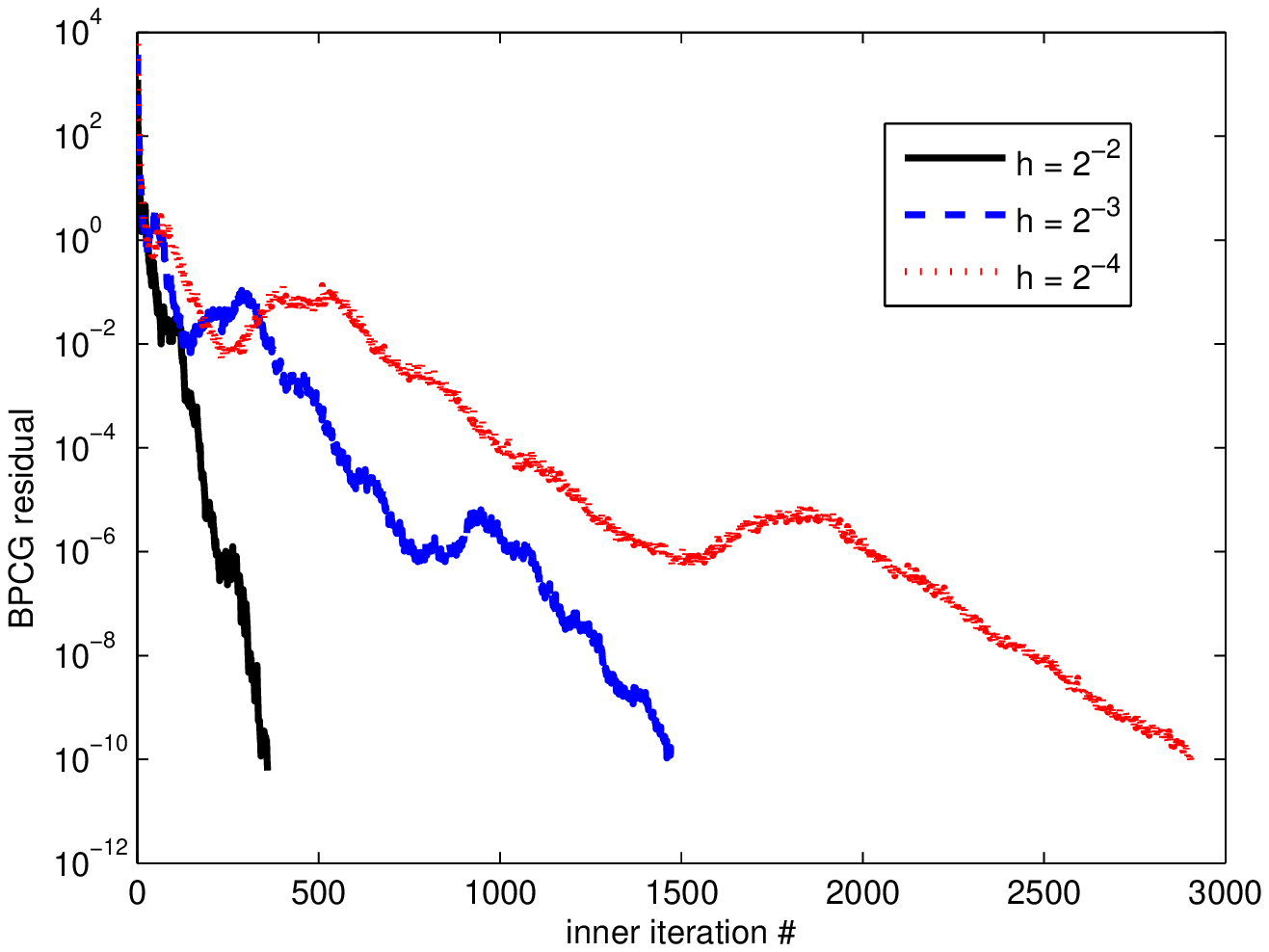}
\caption{Unconstrained problem with convection (problem {\tt CC-pb1} described
in Section \ref{num_exp}).
Left:  CPU time for a single Newton step vs. the 
discretized state space dimension, for $\beta=(\beta_1,0,0)$ with $\beta_1 = 10,100,1000$. 
Right: {\sc bpcg} residual convergence history for various grid
levels ($\beta_1=1000$). 
\label{herzog}}
\end{figure}

In \cite{StollWathen2012} CC problems with a self-adjoint 
and positive definite elliptic operator as constraint is considered.
Differently from \cite{HerzogSachs2010}, at each nonlinear
iteration a saddle point system is obtained by eliminating $(\mu_{k+1})_{\mathcal{A}_k}$
from the system (\ref{new_eq}) and therefore solving 
a system of reduced dimensions with the following coefficient matrix
$$
J_{k,red} =\begin{bmatrix}
M & 0 & -L^T \\ 0 & \nu M_{{\cal A}_k, {\cal A}_k} & M_{{\cal A}_k, :} \\
- L & M_{:,{\cal A}_k} & 0
\end{bmatrix} ,
$$
where $M_{{\cal C}_r, {\cal C}_c}$ is the submatrix of $M$ obtained by taking the rows whose 
indices belong to the set ${\cal C}_r$ and the columns whose indices belong 
to the set ${\cal C}_c$. Here, `:' denotes the set of all indices $1,\ldots,n_h$.
However, the authors of \cite{StollWathen2012} preferred to work with the full $3\times3$ block system, 
\begin{equation}\label{JF}
J_F := \begin{bmatrix}
M & 0 & -L^T \\ 0 & \nu M & M \\
- L & M & 0 
\end{bmatrix} ,
\end{equation}
which they
considered to be more practical to handle within 
the semismooth Newton method, than a system whose full dimension depends
on the number of indices in the active sets.
To solve these complete systems, the following block triangular
preconditioner ${\cal P}^{BT}$ and inner product  matrix ${\cal H}$ are introduced
in \cite{StollWathen2012}:
\begin{equation}\label{BT}
{\cal P}^{BT} =
\begin{bmatrix}
A_0 & 0 & \cdot \\ 0 & A_1 & 0 \\ -L & M & - S_0
\end{bmatrix} ,
\qquad
{\cal H} =
\begin{bmatrix}
M-A_0 & 0 & 0 \\ 0 & \nu M - A_1 & 0 \\ 0 & 0 &  S_0
\end{bmatrix} ,
\end{equation}
where $A_0$ and $A_1$ are appropriate approximations of $M$ and $\nu M$, respectively, so that the
matrix $\cal H$ is positive definite; moreover,
$S_0 =  L M^{-1} L$ approximates the following true
Schur complement of $J_F$:
\begin{equation}\label{SF}
S_F = L M^{-1} L + \nu^{-1} M . 
\end{equation}
Note that the preconditioner 
${\cal P}^{BT}$ does not depend on the nonlinear iteration $k$, and therefore on the current active set.

As in the previous approach, the preconditioned system $\left({\cal P}^{BT}\right)^{-1} J_F$  
is symmetric and positive definite with respect to the inner product
associated with ${\cal H}$ and a CG method can be applied. 
In Section \ref{num_exp} we will report on the performance of the preconditioners
${\cal P}^{BT}$, compared with our new preconditioners.

In our analysis, we found the work \cite{PearsonWathen2012} particularly inspiring, although
a simplified setting was considered: a positive definite self-adjoint
elliptic operator in the equality constraint, and no bound-constraints.
Under these hypotheses, the KKT conditions give a saddle
point system with the coefficient matrix $J_F$ in (\ref{JF}).
In \cite{PearsonWathen2012} 
the following factorized approximation to the Schur complement $S_F$ in (\ref{SF}) is introduced
(the scaling factor $\frac 1 \nu$ is omitted):
\begin{equation}\label{schur-wath}
 \widehat S_F =  (\sqrt{\nu}L+ M) M^{-1} (\sqrt{\nu}L+ M) ,
\end{equation}
which appears to possess nice independence properties 
with respect to the problem parameters:
the eigenvalues of $\widehat S_F^{-1} S_F$ lie in the interval 
$\left[ \frac{1}{2},1\right ]$ independently of the values of $h$ and $\nu$
 \cite[Theorem 4]{PearsonWathen2012}.
In the following we shall broadly generalize this idea so as to cover our more complete 
framework. Optimality results will also be discussed.

The Schur complement approximation (\ref{schur-wath}) was also used in
\cite{PearsonWathen2013} in the solution of convection-diffusion (equality constrained) control problems 
where the authors generalized  the above mentioned spectral properties of $\widehat S_F^{-1} S_F$ to
the case where $L$ is nonsymmetric.

\section{A new approximation to the active-set Schur complement} \label{sec:as_schur}

In agreement with other commonly employed preconditioning strategies, the
preconditioners we are going to present in Section~\ref{sec:prec}  strongly
rely on the quality of the used approximation to the Schur complement 
of the coefficient matrix $J_k$.
In this section we introduce this approximation and analyze its spectral
properties.
%
In the following we shall make great use of the fact that $M$ is a lumped mass matrix,
and thus diagonal. This way, $M$ and $\Pi_{\mathcal {A}_k}$ can commute and formulas simplify
considerably.
To simplify the notation, we shall use the short-hand notation $\Pi_k=\Pi_{\mathcal {A}_k}$.
 
Using the same partitioning as in (\ref{eqn:J_part}),
the {\em active-set Schur complement} associated with $J_k$, and its block factorization
are given by
\begin{eqnarray*} 
S_k  =  B_k A^{-1} B_k^T 
&=& \frac{1}{\nu}\begin{bmatrix}
\nu L M^{-1} L^T + M & (\alpha_y \nu L M^{-1}-\alpha_u I) P_{\mathcal{A}_k}^T  \\
 P_{\mathcal{A}_k}(\alpha_y \nu M^{-1}L^T-\alpha_u I)  &
(\alpha_y^2\nu + \alpha_u^2) P_{\mathcal{A}_k} M^{-1} P_{\mathcal{A}_k}^T 
\end{bmatrix} \\
&=& \frac 1 \nu R_k 
 \begin{bmatrix}
{\mathbb S}_k  & 0 \\ 0 & (\alpha_y^2\nu + \alpha_u^2)  
P_{\mathcal{A}_k} M^{-1} P_{\mathcal{A}_k}^T 
\end{bmatrix}  
  R_k^T ,
\end{eqnarray*} 
with
\begin{equation}\label{Rk}
 R_k = \begin{bmatrix} I & \frac{1}{\alpha_y^2\nu + \alpha_u^2}(\alpha_y \nu L M^{-1} -\alpha_u I)\Pi_k M P_{\mathcal{A}_k}^T \\ 0 & I 
\end{bmatrix},
\end{equation}
and
\begin{equation*}
 {\mathbb S}_k  =   \nu L M^{-1} L^T + M - \frac{1}{\alpha_y^2\nu+\alpha_u^2} 
  (\alpha_y \nu L M^{-1}-\alpha_u I) \Pi_k M \Pi_k (\alpha_y\nu L M^{-1}-\alpha_u I)^T.
\end{equation*}
We define the following {\em factorized approximation} of ${\mathbb S}_k$:
\begin{equation}\label{eqn:shat_gen}
\widehat {\mathbb S}_k : = L_1 M^{-1} L_1^T, \qquad
{\rm with } \quad 
L_1 = \sqrt{\nu}L \left(I-\gamma_1\Pi_k\right)^{\frac 1 2} +
\left(I-\gamma_2\Pi_k \right)^{\frac 1 2} M,
\end{equation}
and 
\begin{equation}\label{gamma12_def}\gamma_1=\frac{\alpha_y^2 \nu}{\alpha_y^2\nu+\alpha_u^2}, \qquad
\gamma_2=\frac{\alpha_u^2}{\alpha_y^2\nu+\alpha_u^2}.
\end{equation}
Note that $\gamma_1 + \gamma_2 = 1$, which implies 
\begin{equation}\label{gamma12}
 (I - \gamma_1 \Pi_k)^{\frac 1 2}(I-\gamma_2 \Pi_k)^{\frac 1 2} 
= \sqrt{\gamma_1\gamma_2}\ \Pi_k +(I- \Pi_k),
\end{equation}
a property that will be used in the sequel.
 Moreover both (diagonal) matrices under square root have strictly 
 positive diagonal elements for $\gamma_1,\gamma_2 \neq 1$, i.e., for  $\alpha_u\neq 0$
 and $\alpha_y\neq 0$, respectively. If $\gamma_1=1$ (or $\gamma_2=1$),
 then  $(I - \gamma_1 \Pi_k)^{\frac 1 2}$ (or $(I - \gamma_2 \Pi_k)^{\frac 1 2}$) reduces to  $(I - \Pi_k)$.
 
\vskip 0.05in
\begin{remark}
{\rm 
Our approach uses the fact that $M$ is diagonal, both from a
computational and a theoretical point of view.
If the employed discretization is such that $M$ is no longer diagonal, 
then we could define the preconditioner with {\rm diag}($M$) in place of $M$.
As an alternative, we could keep $M$ in the preconditioner, and solve systems
with $P_{{\cal A}_k}M^{-1} P_{{\cal A}_k}^T$ as discussed in 
\cite[(3.10)]{HerzogSachs2010}, and possibly
also approximate the action of $M^{-1}$ by a Chebyshev polynomial \cite{Wathen2008}.
For the sake of simplicity we refrain from further exploring these possibilities.
}
\end{remark}

\vskip 0.05in

We proceed with an analysis of the quality of the proposed Schur complement
preconditioner.

\begin{proposition}\label{prop:hatS=SpG_gen}
Let ${\mathbb S}_k$ and $\widehat {\mathbb S}_k$ be as defined above. Then
$$
 \widehat{\mathbb S}_k  =  {\mathbb S}_k + \sqrt{\nu}( L (I- \Pi_k) + (I-\Pi_k) L^T) .
$$
\end{proposition}
\begin{proof}
The result follows from
\begin{eqnarray*}
 {\mathbb S}_k  & = & \nu  L M^{-1} L^T + M - \frac{1}{\alpha_y^2\nu+\alpha_u^2} (\alpha_u^2 \Pi_k M +  \alpha_y^2 \nu^2 L \Pi_k M^{-1}L^T  - \alpha_y \alpha_u \nu ( \Pi_k L^T + L \Pi_k ) ) \\
 & = & \nu  L (I-\gamma_1\Pi_k) M^{-1} L^T + (I-\gamma_2\Pi_k) M +
 \sqrt{\nu} (L \sqrt{\gamma_1 \gamma_2}\ \Pi_k + \sqrt{\gamma_1 \gamma_2}\ \Pi_k L^T),
\end{eqnarray*}
and 
\begin{eqnarray*}
\widehat{\mathbb S}_k & = & (\sqrt{\nu}L \left(I-\gamma_1 \Pi_k\right)^{\frac 1 2} + \left(I-\gamma_2  \Pi_k\right)^{\frac 1 2} M) M^{-1} 
(\sqrt{\nu}L \left(I-\gamma_1 \Pi_k\right)^{\frac 1 2} + \left(I-\gamma_2 \Pi_k\right)^{\frac 1 2} M)^T \\
& = & \nu  L (I-\gamma_1\Pi_k) M^{-1} L^T + (I-\gamma_2\Pi_k) M + \\
& & \sqrt{\nu}L\left(I-\gamma_1 \Pi_k\right)^{\frac 1 2} \left(I-\gamma_2 \Pi_k\right)^{\frac 1 2} 
+ \sqrt{\nu} \left(I-\gamma_1 \Pi_k\right)^{\frac 1 2} \left(I-\gamma_2 \Pi_k\right)^{\frac 1 2} L^T \\
& = & \nu  L (I-\gamma_1\Pi_k) M^{-1} L^T + (I-\gamma_2\Pi_k) M + \\
&& + \sqrt{\nu} L \left(\sqrt{\gamma_1\gamma_2}\ \Pi_k +(I- \Pi_k) \right)+ \sqrt{\nu} \left(
\sqrt{\gamma_1\gamma_2}\ \Pi_k +(I- \Pi_k) \right) L^T,
\end{eqnarray*}
where (\ref{gamma12}) was used.
\end{proof}
\vskip 0.1in
Note that the difference between the true and the approximate Schur complement
does not depend on the $\gamma$'s.
%
The following special case of Proposition \ref{prop:hatS=SpG_gen} occurs
when all indices are active, so that $\Pi_k=I$. 
\vskip 0.05in
\begin{corollary}\label{cor:Afull}
If $\mathcal{A}_k = \{1, \ldots, n_h\}$, then $ \widehat{\mathbb S}_k  = {\mathbb S}_k$.
\end{corollary}
\vskip 0.05in
The Schur complement approximation specializes when particular choices
of $\alpha_u$ and $\alpha_v$ are made.
In the CC case, that is for $(\alpha_u, \alpha_y) = (1,0)$,
we obtain 
$$
L_1= \sqrt{\nu}L +(I-\Pi_k)M.
$$
In the case of $L$ symmetric and no bound constraints, that is for
$\mathcal{A}_k = \emptyset$, we obtain $L_1=\sqrt{\nu}L +M$, which 
corresponds to the factor in (\ref{schur-wath}), as introduced in \cite{PearsonWathen2012}.
%
%
%
%
In the Mixed Constraints case, that is
for $(\alpha_u, \alpha_y) =(\epsilon,1)$, we obtain 
$$
L_1 = \sqrt{\nu}L \left(I-\frac{1}{1+\gamma} \Pi_k\right)^{\frac 1 2} + 
\left(I-\frac{\gamma}{1+\gamma} \Pi_k\right)^{\frac 1 2} M ,
$$
with $\gamma=\epsilon^2/\nu$. Note that both (diagonal) matrices under square root have strictly 
positive diagonal elements for $\gamma>0$.
Finally, in the pure State Constraints case, i.e. for $(\alpha_u, \alpha_y) =(0,1)$, we obtain 
$$
L_1 = \sqrt{\nu}L \left(I-\Pi_k\right) +  M.
$$

In the next proposition we derive general estimates for the
inclusion interval for the eigenvalues
of the pencil $({\mathbb S}_k, \widehat{\mathbb S}_k)$, whose
extremes depend on the spectral properties of the nonsymmetric matrix $L$ and on $M$, for
general $\mathcal{A}_k$. Special cases will then be singled out.
\vskip 0.1in
\begin{proposition}\label{prop:estimatesShatS_gen}
Assume that $\widehat{\mathbb S}_k$ is nonsingular. 
Let 
\begin{equation}\label{G}
G_k := F (I- \Pi_k) + (I-\Pi_k) F^T,
\end{equation}
where $F=\sqrt{\nu} M^{-\frac 1 2} L M^{-\frac 1 2}$, with $F$ nonsingular, 
and
\begin{equation}\label{H}
H_k := F(I-\gamma_1 \Pi_k) F^T + (I-\gamma_2 \Pi_k)+  \sqrt{\gamma_1\gamma_2}(F \Pi_k + \Pi_k F^T),
\end{equation}
with
$\gamma_1,\gamma_2$ as defined in (\ref{gamma12_def}). 
Then
\begin{equation}\label{alphamin}
\alpha_{\min} := \displaystyle\min_{z\ne 0}\frac{z^T G_k z}{z^T H_k z}>-1 ,
\end{equation}
and the eigenvalues $\lambda$ of the pencil $({\mathbb S}_k, \widehat {\mathbb S}_k)$
satisfy $\lambda \in \left [ \frac 1 2, \frac{1}{1+\alpha_{\min}}\right].$
\end{proposition}
\begin{proof}
For the sake of readability, we omit the subscript $k$ within this proof.
The matrix $H$ in (\ref{H}) satisfies $H= M^{-\frac 1 2} {\mathbb S} M^{-\frac 1 2}$.
Let 
\begin{equation}\label{Hhat}
\widehat  H= M^{-\frac 1 2} \widehat {\mathbb S} M^{-\frac 1 2}.
\end{equation}
Then by Proposition~\ref{prop:estimatesShatS_gen} we have that
$G,H$ in (\ref{G})  and (\ref{H}) satisfy
$\widehat  H  =  H +G$.  Therefore the problem
${\mathbb S} x = \lambda \widehat{\mathbb S} x$
can be written as $H z  = \lambda (H+G) z$, with $ z = M^{\frac 1 2 } x$, and for $z\ne 0$ we can write
$$
\lambda = \frac{1}{1+\frac{z^T G z}{z^T H  z}}.
$$
For $z\ne 0$ we have 
$\frac{z^T G z}{z^T H  z} > -1$ if and only if
$z^T (G + H) z >0$.
The latter inequality is satisfied since
$ G + H = M^{-\frac 1 2} \widehat{\mathbb S} M^{-\frac 1 2} $, and
$\widehat{\mathbb S}$ is positive definite.
This proves the upper bound for $\lambda$.

To prove the lower bound, we  first consider the case $\gamma_2\neq 1$.
We define $W:=(I-\gamma_2 \Pi)^{-\frac 1 2} F (I-\gamma_1 \Pi)^{\frac 1 2}$
and notice that
$$
(I-\gamma_2 \Pi)^{-\frac 1 2} \widehat{H} (I-\gamma_2 \Pi)^{-\frac 1 2} =
(W+I)(W+I)^T ,
$$
while
\begin{eqnarray*}
&&(I-\gamma_2 \Pi)^{-\frac 1 2} H(I-\gamma_2 \Pi)^{-\frac 1 2} \\
&=& WW^T + I + 
\sqrt{\gamma_1\gamma_2} \left ( W\Pi (I - \gamma_1 \Pi)^{-\frac 1 2}(I-\gamma_2 \Pi)^{-\frac 1 2}
+ (I-\gamma_2 \Pi)^{-\frac 1 2} (I - \gamma_1 \Pi)^{-\frac 1 2} \Pi W^T\right ) \\
&=&
WW^T + I + (W\Pi+\Pi W^T),
\end{eqnarray*}
where the relation (\ref{gamma12})  was used.
For $x\ne 0$ we can thus write
$$
\lambda = 
\frac{ x^T {\mathbb S}x}{x^T \widehat{\mathbb S} x} =
\frac{ y^T (WW^T + I + (W\Pi+\Pi W^T))y}{y^T (W+I)(W+I)^T  y},
$$
where $y=(I-\gamma_2 \Pi)^{\frac 1 2} M ^ {\frac 1 2 }x$. Therefore, $\lambda \ge \frac 1 2$ if and only
if 
$$
\frac{ y^T (WW^T + I + (W\Pi+\Pi W^T))y}{y^T (W+I)(W+I)^T  y} \ge \frac 1 2 \qquad
$$
which is equivalent to
$$
\frac 1 2 y^T( WW^T + I + W (2\Pi - I) + (2\Pi - I)W^T) y \ge 0.
$$
Noticing that $I=(2\Pi - I)(2\Pi - I)$,
it holds 
$$WW^T + I + W (2\Pi - I) + (2\Pi - I)W^T = (W+ (2\Pi-I)) (W+ (2\Pi-I))^T \succeq  0;
$$
Therefore, the last inequality is always verified, proving the lower bound for $\lambda$.

Consider now the case  $\gamma_2 = 1$ (which implies  $\gamma_1= 0$).
We define $W:= F^{-1} (I-\Pi)$.
For $x\ne 0$ we can  write
$$
\lambda = 
\frac{ x^T {\mathbb S}x}{x^T \widehat{\mathbb S} x} = 
\frac{ z^T (FF^T + (I-\Pi) )z}{z^T (F+(I-\Pi))(F+(I-\Pi))^T  z} =
\frac{ y^T (WW^T + I )y}{y^T (W+I)(W+I)^T  y},
$$
where $z = M ^ {\frac 1 2 }x$ and $y= F^T z$. As above, 
$\lambda \ge \frac 1 2$ if and only
if 
$$
\frac{ y^T (WW^T + I )y}{y^T (W+I)(W+I)^T  y} \ge \frac 1 2 
$$
which holds since $ 2(WW^T + I) - (W+I)(W+I)^T= (W-I)(W-I)^T \succeq  0$.
\end{proof}
\vskip 0.1in

Proposition \ref{prop:estimatesShatS_gen} reformulates the eigenvalue
problem with the preconditioned Schur complement in terms of
the eigenvalue problem with a different Rayleigh quotient, which 
seems to be easier to interpret. 
Numerical experiments confirm the sharpness of
the lower extreme (see below); for the upper bound more insightful estimates
can be given under additional hypotheses, and these are explored in
the following.

\vskip 0.05in
\begin{corollary}\label{cor:Aempty_gen}{\rm \cite[Theorem 4.1]{PearsonWathen2013}}
Assume $L + L^T \succeq 0$ and let $\mathcal{A}_k = \emptyset$.
Then the eigenvalues $\lambda$ of the pencil $({\mathbb S}_k, \widehat {\mathbb S}_k)$
satisfy
$\lambda \in \left[ \frac{1}{2} , 1 \right]$.
\end{corollary}

\vskip 0.05in
The result of Corollary  \ref{cor:Aempty_gen} generalizes the result of \cite{PearsonWathen2012}
to nonsymmetric and positive semidefinite $L$, showing the optimality and robustness of
the approximation with respect to the problem parameters. 

To be able to analyze another interesting special case, we first need
an auxiliary lemma whose proof is postponed to the appendix. 

\begin{lemma}\label{lemma:F}
Let $F\in{\mathbb R}^{n\times n}$ be such that $F+F^T\succeq0$.  Then

i) $\|(F+I)^{-1}(F-I)\| \le 1;$

ii) $\|(F+I)^{-1}(F+F^T)(F+I)^{-T}\| \le \frac 1 2$.
\end{lemma}
\vskip 0.05in
%
%
 
We can now estimate the eigenvalues of $({\mathbb S}_k, \widehat {\mathbb S}_k)$ for a particular
choice of $\gamma_1, \gamma_2$.

\vskip 0.05in
\begin{proposition} \label{prop:gamma12}
Assume $L+L^T\succeq0$ and let $\gamma_1=\gamma_2=\frac 1 2$. Then 
the eigenvalues $\lambda$ of the pencil $({\mathbb S}_k, \widehat {\mathbb S}_k)$ satisfy
$\lambda \in \left[\frac 1 2\,, \,3\right]$.
\end{proposition}

\vskip 0.05in
\begin{proof} 
For the sake of readability, we omit the subscript $k$ within this proof.
We only have to prove the upper bound. Let $F=\sqrt{\nu}M^{-\frac 1 2} L M^{-\frac 1 2}$, so that $F+F^T\succeq0$.
Proceeding as in the proof of Proposition \ref{prop:estimatesShatS_gen},
the eigenproblem ${\mathbb S} x = \lambda \widehat {\mathbb S} x$ can be transformed into
\begin{equation}\label{eqn:PbF}
 H y = \lambda (H+G) y ,
\end{equation}
with $y= M^{\frac 1 2} x$, where $H$ and $G$ are given in (\ref{H}) and (\ref{G}), respectively.
For $\gamma_1=\gamma_2=\frac 1 2$, we have
$ H +G  = (F+I)\left(I-\frac 1 2 \Pi\right)(F+I)^T$, while
$H = (F-I)\left(I-\frac 1 2 \Pi\right)(F-I)^T + F + F^T$, which can be readily verified.
Therefore, problem (\ref{eqn:PbF}) can be written as
$$
\left((F-I)\left(I-\frac 1 2 \Pi\right)(F-I)^T + F + F^T\right) y = \lambda (F+I)\left(I-\frac 1 2 \Pi\right)(F+I)^T y ,
$$
or equivalently, with $u = (F+I)^T y$, as
\begin{eqnarray}\label{eqn:newe}
(F+I)^{-1}  \left ( (F-I)\left(I-\frac 1 2 \Pi\right)(F-I)^T +  F + F^T \right ) (F+I)^{-T} u = && 
\hskip 0.6in \nonumber \\
  \lambda (I-  \frac 1 2  &&\Pi  ) u .  
\end{eqnarray}
We then multiply (\ref{eqn:newe}) from the left by $u^T\ne 0$,
\begin{eqnarray}\label{eqn:newe1}
u^T (F+I)^{-1}  \left ( (F-I)\left(I-\frac 1 2 \Pi\right)(F-I)^T +  F + F^T \right ) (F+I)^{-T} u = &&
\hskip 0.6in \nonumber \\
 \lambda u^T(I-  \frac 1 2 && \Pi) u ,
\end{eqnarray}
 and
we note that $u^T \left(I-\frac 1 2 \Pi\right) u \ge \frac 1 2 \|u\|^2$. Moreover, 
using Lemma~\ref{lemma:F}
\begin{eqnarray*}
u^T (F+I)^{-1} (F-I) && \!\!\!\!\! (I-\frac 1 2 \Pi)(F-I)^T (F+I)^{-T} u  \\
&\le & \|\left(I-\frac 1 2 \Pi\right)\|  \|(F-I)^T (F+I)^{-T}\|^2 \|u\|^2 \le \|u\|^2 ,
\end{eqnarray*}
and
$u^T (F+I)^{-1}(F + F^T) (F+I)^{-T} u \leq \frac{1}{2} \|u\|^2$.  
Therefore, 
using these last bounds in (\ref{eqn:newe1}) we obtain
$\|u\|^2 + \frac 1 2\|u\|^2 \ge  \lambda \frac 1 2 \|u\|^2$, with $\|u\| \ne 0$,
from which the upper estimate follows.
\end{proof}

In the notation of Proposition \ref{prop:estimatesShatS_gen}, for $\gamma_1=\gamma_2$
we bounded $\alpha_{\min}$ by $-\frac 2 3$.

\vskip 0.1in
\begin{remark}\label{rem:nueps2}
{\rm
The case $\gamma_1=\gamma_2$ comprises MC problems where $\alpha_u=\epsilon$ and
$\alpha_y = 1$, so that the equality $\nu=\epsilon^2$ holds. Therefore, for $\nu=\epsilon^2$ 
Proposition \ref{prop:gamma12} ensures 
a clustered spectrum of the preconditioned Schur complement, and this also strongly
influences the spectrum of the overall preconditioned matrix - see Section \ref{sec:prec} - 
predicting fast convergence of the iterative methods. 
From an application perspective,
these experiments show that if $\nu \approx \epsilon^2$ in the given model, then
a good performance of the solver is expected.
}
\end{remark}

\vskip 0.1in
The good behavior for $\nu=\epsilon^2$ discussed in the remark above
is confirmed by our numerical experiments (see Example~\ref{ex:2}), where
problems with MC constraints (\ref{pb_mc}) are tested for all combinations of values of $\nu$ and $\epsilon$:
the best performance is indeed obtained for $\nu=\epsilon^2$. It is also interesting
to observe that our findings are in agreement with similar experimental observations
reported in \cite{Borzi2007}, where the case $\nu \approx \epsilon^2$ ensured the
best performance of a multigrid solver for the MC problem.

Tables \ref{eigS}-\ref{eigSSC} display the spectral intervals for ${ \widehat{\mathbb S}_k}^{-1} {\mathbb S}_k $ for
the three considered model problems (see Table \ref{table:testpbs}).
In all tables,
the minimum and maximum eigenvalues 
are reported for
the $k$th iteration for which $\lambda_{\max}({ \widehat{\mathbb S}_k}^{-1} {\mathbb S}_k)$ is maximum.
The CC case shows the largest, though still extremely modest, dependence of $\lambda_{\max}$ on the
problem parameters, and this dependence quickly fades as $\beta_1$ increases.
On the other hand, $\lambda_{\min}$ remains largely insensitive to parameter variations, with a
small benign increase from the bound $\frac 1 2$ for $\nu=10^{-2}$ as $\beta_1$ grows.
In the mixed case and $\nu=\epsilon^2$, $\lambda_{\max}$ remains well below the upper estimate $3$, for
a variety of mesh parameter values.

\begin{table}[htb]
\begin{center}
\footnotesize
\begin{tabular}{|c|r|crrr|crrr|}
\hline
 &     &            \multicolumn{ 4}{c|}{$\nu = 10^{-2}$} &            \multicolumn{ 4}{c|}{$\nu = 10^{-6}$} \\
				\hline
$\beta_1$ &     $h$  &        $k$ & $|\mathcal{I}^k|$ & $\lambda_{\min}$ & $\lambda_{\max}$ &        $k$ & $|\mathcal{I}^k|$ & $\lambda_{\min}$ & $\lambda_{\max}$ \\
\hline
0 &\multicolumn{ 1}{c|}{$ 2^{-2}$} &   1 &         98 &       0.51 &       1.24 &          3 &         25 &       0.55 &        4.7 \\
&\multicolumn{ 1}{c|}{$ 2^{-3}$} &       3 &        895 &       0.51 &       1.27 &         17 &         24 &        0.5 &      13.14 \\
\hline

10 &\multicolumn{ 1}{c|}{$ 2^{-2}$} &       1 &         73 &       0.64 &       1.18 &          5 &         57 &       0.54 &       5.32 \\
&\multicolumn{ 1}{c|}{$ 2^{-3}$} &     1 &        891 &       0.61 &       1.24 &          6 &         44 &        0.50 &      10.72 \\
\hline

100 &\multicolumn{ 1}{c|}{$ 2^{-2}$} &    1*&          0 &          1 &          1 &          4 &         49 &       0.51 &       4.82 \\
&\multicolumn{ 1}{c|}{$ 2^{-3}$} &        1 &        120 &       0.95 &       1.01 &          6 &        201 &        0.5 &       6.64 \\

\hline
1000 &\multicolumn{ 1}{c|}{$ 2^{-2}$} &      1*&          0 &          1 &          1 &          2 &         49 &        0.6 &       1.39 \\
&\multicolumn{ 1}{c|}{$ 2^{-3}$} &     1*&          0 &          1 &          1 &          2 &        675 &       0.58 &       1.63 \\

\hline 
\end{tabular}\end{center}
\hskip 0.66in {\footnotesize * Newton terminates in 2 steps.}
\caption{
Control-Constraints: Extreme eigenvalues of ${ \widehat{\mathbb S}_k}^{-1} {\mathbb S}_k $,
Newton iteration $k$, and dimension of the Inactive set, $|\mathcal{I}^k|$, as the mesh size $h$, the 
regularization parameter $\nu$ and the convection parameter $\beta=(\beta_1,0,0)$  vary.}
 \label{eigS}
 \end{table}

\begin{table}[htb]
\begin{center}
\footnotesize
\begin{tabular}{|c|r|ccrrr|ccrrr|}
\hline
   &      &      \multicolumn{ 5}{c|}{$\nu = 10^{-2}$} &            \multicolumn{ 5}{c|}{$\nu = 10^{-6}$} \\
					\hline

  $\beta_1$        & $h$ &        $\epsilon$ & $k$&$|\mathcal{I}^k|$ & $\lambda_{\min}$ & $\lambda_{\max}$ &        $\epsilon$ & $k$ &$|\mathcal{I}^k|$ & $\lambda_{\min}$ & $\lambda_{\max}$ \\
\hline
0  & $2^{-2}$                  & $10^{-1}$  & 1&  196 &       0.53 &       1.10 &          $10^{-1}$  & 2&         165 &       0.64 &    1.97 \\
  &                   & $10^{-2}$  & 1&  294 &       0.51 &       1.51 &          $10^{-2}$  & 1&         196 &       0.75 &    1.10 \\
  & &    $10^{-3}$ &     2&   303 &       0.50 &        1.93 &          $10^{-3}$  & 1&     196&       0.75 &       1.01 \\
  & {$2^{-3}$}           & $10^{-1}$  & 2&         2242 &       0.52 &       1.16 &          $10^{-1}$  & 3&       1212 &       0.51 &        2.63 \\
   &           & $10^{-2}$  & 3&         2782 &       0.51 &       1.97 &          $10^{-2}$  & 1&       1800 &       0.51 &        1.29 \\
 & &   $10^{-3}$ &  3&     3030 &       0.51 &       3.37 &          $10^{-3}$  & 1&  1800 &       0.51 &       1.03 \\
\hline
10  & $2^{-2}$   & $10^{-1}$  &    1&     315 &    0.53 &     1.05 &          $10^{-1}$  &  2&     147 &       0.57 &       3.28 \\
    &                  & $10^{-2}$  &  0*&   343 &       0.53 &     0.93    &          $10^{-2}$  & 1&         196 &       0.69 &    1.37 \\
& &    $10^{-3}$ &    0*&   343 &       0.53 &     0.93  &          $10^{-3}$  &  1&196 &       0.69 &       1.03 \\
  & {$2^{-3}$} & $10^{-1}$   &     1&   2549 &       0.56 &       1.18 &          $10^{-1}$  &  3&     1406 &       0.50 &       5.20 \\
  & & $10^{-2}$   &     1&   3135 &       0.53 &       1.55 &          $10^{-2}$  &  2&     1631 &       0.50 &        1.71 \\
&&   $10^{-3}$ &     1&  3303 &0.53 &     2.45 &          $10^{-3}$  & 1&1800&       0.50 &       1.08 \\
\hline

100  & $2^{-2}$  & $10^{-1}$  &    0*&      343 &    0.84 &   0.98 &          $10^{-1}$  & 2&        196 &       0.51 &        4.40 \\
     &   & $10^{-2}$  &    0*&      343 &    0.84 &   0.98 &          $10^{-2}$  & 2&        196 &       0.51 &        2.49 \\
&&    $10^{-3}$ &   0*&    343 &     0.84 &   0.98 &          $10^{-3}$  &    1&   147 &       0.51 &       1.24 \\
  & {$2^{-3}$} & $10^{-1}$  &   1&        3299 &     0.84 &       1.01  &          $10^{-1}$  & 2&       1575 &       0.50 &       5.9 \\
  &  & $10^{-2}$  &   1&        3367 &     0.84 &       1.12  &          $10^{-2}$  & 2&       1519 &       0.50 &       3.25 \\
&&   $10^{-3}$ &   0*& 3375 &      0.84 &   0.99 &          $10^{-3}$  &  2&  1800     &       0.50&      1.42 \\

\hline
1000  & $2^{-2}$  &   $10^{-1}$  &   0*&      343 &    0.98 &  0.99  &          $10^{-1}$  &  1&      294 &       0.51 &       1.22 \\
  &   &   $10^{-2}$  &   0*&      343 &    0.98 &  0.99  &          $10^{-2}$  &  1&      294 &       0.51 &       1.22 \\
&&    $10^{-3}$ &   0*&    343 &        0.98 &  0.99&          $10^{-3}$  &    1&    294 &       0.52 &       1.24 \\
& $2^{-3}$ & $10^{-1}$  &  0*&        3375 &        0.98 &  0.99 &          $10^{-1}$  &   2 &      2475 &       0.52 &        1.40 \\
&  & $10^{-2}$  &  0*&        3375 &        0.98 &  0.99 &          $10^{-2}$  &   2 &      2644 &       0.52 &        1.34 \\
& &   $10^{-3}$ & 0*& 3375 &        0.98 &  0.99 &          $10^{-3}$  &   2&  2925 &       0.51 &       1.31 \\
\hline
\end{tabular} \end{center}
\hskip 0.25in {\footnotesize * Newton terminates in 1 step.}
\caption{Mixed-Constraints: Extreme eigenvalues of ${ \widehat{\mathbb S}_k}^{-1} {\mathbb S}_k $,
Newton iteration $k$, and dimension of the Inactive set, $|\mathcal{I}^k|$, as the mesh size $h$, the 
regularization parameters $\nu, \epsilon$ and the convection parameter $\beta=(\beta_1,0,0)$ vary.}\label{eigSMC}
 \end{table}

\begin{table}[htb]
\begin{center}
\footnotesize
\begin{tabular}{|c|r|crrr|crrr|}
\hline
 &     &            \multicolumn{ 4}{c|}{$\nu = 10^{-2}$} &            \multicolumn{ 4}{c|}{$\nu = 10^{-6}$} \\
				\hline
$\beta_1$ &     $h$  &        $k$ & $|\mathcal{I}^k|$ & $\lambda_{\min}$ & $\lambda_{\max}$ &        $k$ & $|\mathcal{I}^k|$ & $\lambda_{\min}$ & $\lambda_{\max}$ \\
\hline
0 &\multicolumn{ 1}{c|}{$ 2^{-2}$} &          2 &         303 &       0.50 &       2.01  &          1 &        196&       0.75 &      1.00 \\
&\multicolumn{ 1}{c|}{$ 2^{-3}$} &            3 &       3030 &       0.51 &       3.65 &         1 &         1800 &      0.51 &      1.02 \\
\hline

10 &\multicolumn{ 1}{c|}{$ 2^{-2}$} &          0* &        343 &       0.53 &       0.93 &          1 &       196 &       0.69 &       1.02 \\

&\multicolumn{ 1}{c|}{$ 2^{-3}$} &          1 &       3319 &       0.52 &       2.94 &          1 &       1800 &          0.50 &       1.06  \\

\hline
100 &\multicolumn{ 1}{c|}{$ 2^{-2}$} &          0* &        343 &       0.84 &       0.98 &      1 &        196 &        0.50 &       1.23 \\

&\multicolumn{ 1}{c|}{$ 2^{-3}$} &          0* &       3375 &       0.84 &       0.99 &          2 &       2250 &        0.50 &       1.56 \\

\hline
1000 &\multicolumn{ 1}{c|}{$ 2^{-2}$} &          0* &        343 &       0.98 &       0.99 &          0* &        343 &        0.51 &       0.84 \\

&\multicolumn{ 1}{c|}{$ 2^{-3}$} &          0* &       3375 &       0.98 &       0.99 &          0* &       3375 &        0.51 &       0.91 \\
\hline 
\end{tabular}\end{center}
\hskip 0.66in {\footnotesize * Newton terminates in 1 step.}
\caption{
State-Constraints: Extreme eigenvalues of ${ \widehat{\mathbb S}_k}^{-1} {\mathbb S}_k $,
Newton iteration $k$, and dimension of the Inactive set, $|\mathcal{I}^k|$, as the mesh size $h$, the 
regularization parameter $\nu$ and the convection parameter $\beta=(\beta_1,0,0)$  vary.}
 \label{eigSSC}
 \end{table}

The dependence of $\lambda_{\max}$ on the parameters in the CC and SC cases can be analyzed by
using the following result, whose proof is postponed to the appendix. 

\begin{proposition}\label{prop:upperestimate}
Let $\lambda$ be an eigenvalue of $\widehat {\mathbb S}_k^{-1} {\mathbb S}_k$.
Then in the CC and SC case it holds
$$
\lambda \le \zeta^2 + (1 + \zeta)^2 ,
$$
with

\noindent
i) If $(\alpha_u, \alpha_y)=(1,0)$ (CC case), then
$\zeta = \|M^{\frac 1 2} \left (\sqrt{\nu} L  +M (I-\Pi)\right )^{-1} \sqrt{\nu} L M^{-\frac 1 2} \|$;
Moreover, if $L+L^T\succ 0$, then
for $\nu\to 0$, $\zeta$ is bounded by a constant independent of $\nu$;

\noindent
ii) If $(\alpha_u, \alpha_y)=(0,1)$ (SC case), then
$\zeta =\|(I+\sqrt{\nu}M^{-\frac 1 2}LM^{-\frac 1 2}(I-\Pi_k))^{-1}\|$;
Moreover, $\zeta\to 1$ for $\nu\to 0$.
%
%
%
\end{proposition}
\vskip 0.1in

The boundedness of $\zeta$ as $\nu \to 0$  
in both the CC and SC cases justifies the 
good behavior of the eigenvalues shown in Tables~\ref{eigS} and \ref{eigSSC}.  


\section{New preconditioners for the active-set Newton method}\label{sec:prec}
In this section we propose two classes of preconditioners, which can be used
throughout the nonlinear iterations, and automatically modified
 as the system dimensions dynamically
change due to the different number of active indices. 
More precisely, for the problem partitioned as in
(\ref{eqn:J_part}) we consider the following block diagonal preconditioner
$\pd$, and  indefinite preconditioner 
$\pc$:
\begin{equation}\label{BDF}
  \pd = 
\begin{bmatrix}
A & 0   \\
0 & \widehat{S}_k 
\end{bmatrix},
\qquad
\end{equation}
and
\begin{equation}\label{CPF}
  \pc = \begin{bmatrix}
I & 0   \\
B_k A^{-1}  & I 
\end{bmatrix}
\begin{bmatrix}
A & 0   \\
0 & - \widehat{S}_k 
\end{bmatrix}
\begin{bmatrix}
I & A^{-1} B_k^T  \\
0 & I 
\end{bmatrix}, 
\end{equation}
where in both cases, the matrix $\widehat S_k$ is factorized as
\begin{eqnarray*}
\widehat S_k   & = & 
\frac{1}{\nu} R_k \begin{bmatrix} \widehat {\mathbb S}_k  & 0 \\ 
 0 & (\alpha_y^2\nu + \alpha_u^2)  P_{\mathcal{A}_k} M^{-1} P_{\mathcal{A}_k}^T \end{bmatrix}  
  R_k^T,
  \end{eqnarray*} 
with $\widehat{\mathbb S}_k =  L_1 M^{-1} L_1^T$,
and $R_k$ and $L_1$ given in (\ref{Rk}) and (\ref{eqn:shat_gen}), respectively.
%
The following result can be readily proved from Proposition \ref{prop:estimatesShatS_gen}.

\begin{proposition}\label{prop:estimatesPJ_gen}
Assume that $\widehat{\mathbb S}_k$ is nonsingular and  let  $\alpha_{\min}$ be
 as defined in  (\ref{alphamin}).
Then the eigenvalues $\lambda$ of the pencil $(J_k, \pd)$
satisfy
$$ 
\lambda\left(J_k , \; \mathcal{P}_k^{BDF}\right) \in 
\left\{1 , \; \frac{1 \pm \sqrt{5}}{2}\right\} \cup I^- \cup I^+ ,
$$
where
{\footnotesize
$$ 
I^- = \left[\frac{1}{2} \left( 1 - \sqrt{1 + \frac{4}{\left(1 + \alpha_{\min}\right)^2 } } \right), \; 
\frac{1 - \sqrt{2}}{2}\right], 
\,\,
I^+ = \left[ \frac{1 + \sqrt{2}}{2} , \; 
\frac{1}{2} \left( 1 + \sqrt{1 + \frac{4}{\left(1 + \alpha_{\min}\right)^2 } } \right) \right]. 
$$
}
The eigenvalues $\lambda$ of the pencil $(J_k, \pc)$ satisfy
$$
\lambda(J_k, \pc) \in \{1\}\cup \left [\frac{1}{2}, \frac{1}{1+\alpha_{\min}} \right ].
$$
\end{proposition}

\begin{proof}
We observe that the pencil $\left(J_k , \; \mathcal{P}_k^{BDF}\right)$ has the same eigenvalues as:

$$ 
\left(\mathcal{P}_k^{BDF}\right)^{-1/2} J_k \left(\mathcal{P}_k^{BDF}\right)^{-1/2} = 
\begin{bmatrix} I & A^{-1/2} B_k^T \widehat{S}^{-1/2}_k \\ \widehat{S}^{-1/2}_k B_k A^{-1/2} & 0 \end{bmatrix}.
$$

Using \cite[Lemma 2.1]{Fischer1998}, the eigenvalues of the pencil $(J_k , \; \mathcal{P}_k^{BDF})$ 
are either 1 or have the form $ \frac 1 2 \left(1 \pm \sqrt{ 1 + 4 \sigma^2 } \right)$, where $\sigma$ is a singular value 
of $\widehat{S}^{-1/2}_k B_k A^{-1/2}$, that is, $\sigma^2$ is an eigenvalue of $\widehat{S}_k^{-1} S_k$. 
Considering that ${\rm spec}\left(\widehat{S}_k^{-1} S_k\right) = 
\left\{1\right\} \cup {\rm spec}\left(\widehat{\mathbb{S}}_k^{-1} \mathbb{S}_k\right)$, we have
$$ 
\lambda\left( J_k , \; \mathcal{P}_k^{BDF} \right) \in \left\{1, \frac{1 \pm \sqrt{5} }{2} \right\} \cup \left\{ \frac{1}{2} \left(1 \pm \sqrt{ 1 + 4 \sigma^2 } \right) \left.\right| \; \sigma^2 \in {\rm spec} \left(\widehat{\mathbb{S}}_k^{-1} \mathbb{S}_k \right)\right\}.
$$
The claim thus follows from Proposition 4.3.

As for the pencil $\left(J_k , \; {\cal P}_k^{IPF}\right)$, we have the factorization
\begin{eqnarray}\label{eqn:CP}
\left({\cal P}_k^{IPF}\right)^{-1} J_k = 
\begin{bmatrix}  
I & -A^{-1} B_k \\ 0 & I 
\end{bmatrix}
\begin{bmatrix} 
I & 0 \\ 0 & \widehat{S}_k^{-1} S_k \end{bmatrix} \begin{bmatrix} I & A^{-1} B_k \\ 0 & I  
\end{bmatrix}  .
\end{eqnarray}
Again, the result follows from Proposition~\ref{prop:estimatesShatS_gen}.
\end{proof}

Under the stated hypotheses,
refined bounds for the eigenvalues of the  indefinitely preconditioned problem
can be derived using the bounds 
for the eigenvalues of ${\widehat {\mathbb S}_k}^{-1} {\mathbb S}_k$ obtained
in Corollary~\ref{cor:Aempty_gen}, Proposition~\ref{prop:gamma12} and Proposition~\ref{prop:upperestimate}.

In the case of indefinite preconditioning, the preconditioned matrix $\left({\cal P}_k^{IPF}\right)^{-1}J_k$
has real spectrum, however it
is no longer symmetric so that in general, a nonsymmetric solver needs to be applied.
In our numerical experiments we used {\sc gmres} \cite{GMRES}, for which it is known
that the eigenvalues alone may not be sufficient to predict convergence, but that also
eigenvectors play a role. In addition, indefinite preconditioners
are often plagued by the presence of Jordan blocks, whose sensitivity may influence
the use of inexact strategies; see, e.g., \cite{Sesana.Simoncini.13} for a detailed discussion.
Fortunately, since the (1,1) block of $J_k$ is reproduced exactly in the preconditioner,
in our setting the spectral structure is considerably simplified, and in particular,
Jordan blocks do not occur. The following proposition determines the complete
eigenvector decomposition of the preconditioned matrix.

\begin{proposition}\label{prop:invspace}
Let $\widehat S_k^{-1} S_k X = X \Lambda$ be the eigendecomposition of $\widehat S_k^{-1} S_k$,
with $X=[X_1,X_2]$ and $\Lambda={\rm blkdiag}(I,\Lambda_2)$ 
 partitioned so that $X_1$ contains the eigenvectors corresponding to
the unit eigenvalue. Then the preconditioned matrix $\left({\cal P}_k^{IPF}\right)^{-1} J_k$
admits the following eigenvalue decomposition
$$
\left({\cal P}_k^{IPF}\right)^{-1} J_k =
Q \begin{bmatrix} I &  & \\ & I & \\ & & \Lambda_2\end{bmatrix} Q^{-1},
$$
with
$$
Q = 
\left [ \begin{array}{c|c|c}
I & 0   & -A^{-1}B_k X_2 \\
0 & X_1 & X_2 
\end{array}\right ]  ,
\qquad
Q^{-1} = 
\left [ \begin{array}{c|c}
I & A^{-1}B_k X_2 X_2^T \widehat S_k \\ 
0 & X_1^T \widehat S_k \\ 
0 & X_2^T \widehat S_k 
\end{array}\right ]. 
$$
\end{proposition}

\begin{proof} 
Writing
$$
\left({\cal P}_k^{IPF}\right)^{-1} J_k =
\begin{bmatrix}
I & A^{-1}B_k (I-\widehat S_k^{-1} S_k) \\
0 & \widehat S_k^{-1} S_k
\end{bmatrix} ,
$$
the decomposition can be explicitly verified upon substitution. 
The nonsingularity of $Q$ follows from that of $X=[X_1, X_2]$.
The inverse of $Q$ can be derived by observing that $X$ can be chosen so
that $X^T \widehat S_k X = I$.
\end{proof}

The explicit form of Proposition \ref{prop:invspace} allows one to use standard results
to bound the {\sc gmres} residual norm,
by providing bounds for the norm of $Q$ and its inverse $Q^{-1}$, and
exploiting the fact that the spectrum of the preconditioned matrix is real
(see, e.g., \cite[Prop. 6.32]{Saad-book}).

\section{Numerical experiments}\label{num_exp}
In this section we provide a detailed performance analysis
of the proposed preconditioners $\pc$   in (\ref{CPF}) and $\pd$ in (\ref{BDF}) 
for the active-set Newton method and use problems with constraints in (\ref{pb_cc})-(\ref{pb_sc}) as prototypical problems. 
In particular, the analysis of the pure State Constraints case (\ref{pb_sc}) will
be analyzed as the limit case of the MC constraints (\ref{pb_mc}) for $\epsilon \rightarrow 0$.

\begin{table}[htb]
\centering
\begin{tabular}{r|c|c|c|l}
label        & $\Omega$   &  $a$ & $b$ & $y_d$ \\ \hline
{\tt CC-Pb1} & $(-1,1)^3$ &  $0$   & $2.5$ & $1$ for $|x_1|\le \frac 1 2$, $-2$ otherwise \\
{\tt CC-Pb2} & $(0,1)^3$  &  $\frac 1 {10} \exp( -\|x\|^2)$ & $\frac 1 2$ & $\exp(-64 \|x-\frac 1 2\|^2)$ \\
{\tt MC-Pb1} & $(-1,1)^3$ &  $-\infty$ & $0$ & $1$ for $|x_1|\le \frac 1 2$, $-2$ otherwise \\ \hline
\end{tabular}
\caption{Problem data for the numerical experiments. Here $x=(x_1,x_2,x_3)\in\Omega$. \label{table:testpbs}}
\end{table}

In all our examples, we use the three-dimensional data for the discretized problem generated by
the codes in \cite{HerzogSachs2010}. 
The matrices stem from the discretization by
upwind finite differences on a uniform three-dimensional grid (so that $L + L^T \succ 0$).
Zero Dirichlet boundary conditions, that is $\bar y =0$ in (\ref{pb_gen}), were used
throughout.
{\color{black} We stress that for large convection and in the presence of boundary layers,
other discretizations may be more suitable; we refer to the recent nice essay by Stynes on the 
pros and cons of different approaches \cite{Stynes.13}. We also notice that different discretization
techniques will lead to coefficient matrices $L$ with possibly quite different spectral properties.}

In Table \ref{table:testpbs} information on the data used in our numerical experiments can be found, for
two test cases with control constraints, and one test case for 
mixed and state constraints; here $x=(x_1,x_2,x_3)$ is an element of $\Omega$.
The mesh parameter in each direction was taken as
$h \in \{2^{-2},2^{-3},2^{-4},2^{-5}\}$ which corresponds to 
a dimension for the state or control vectors 
$n_h \in \{343, 3375, 29791, 250047\}$. The total linear system dimension
is thus between $3 n_h$ and $4 n_h$, depending on the number of indices in the active set
at each Newton iteration.

 \subsection{Algorithmic considerations} \label{sec:algo}
Throughout this section we consider the implementation
of the active-set Newton method with the following solvers and preconditioning strategies:

\vskip 0.1in
\begin{tabular}{ll}
\ascp& Active-set Newton method with
linear solver {\sc gmres}  \\
& preconditioned with $\pc$; \\
\asbd& Active-set Newton method with
linear solver {\sc minres} \\
& preconditioned with $\pd$; \\
{\asbt}& Variant of active-set Newton method as proposed in \cite{StollWathen2012}, \\
& with {\sc bpcg} preconditioned with ${\cal P}^{BT}$
defined in (\ref{BT}). \\
\end{tabular}
\vskip 0.1in

The application of the Schur complement approximation $\widehat {\mathbb S}_k$ requires solving
with $L_1$ and its transpose in (\ref{eqn:shat_gen}). These solves were replaced by the use of an algebraic multigrid
operator ({\sc hsl-mi20}, \cite{HSL-MI20}), which needs to be 
recomputed at each Newton iteration. 
{\sc hsl-mi20} is used with all default parameters except for the value
{\tt control.st\_parameter=$10^{-4}$}. Moreover, we set the number of pre/post smoothing steps 
equal to 5 for all the experiments with the MC problems, while  with CC problems only
for the finest mesh $h=2^{-5}$.  Although in most cases satisfactory results were obtained with
this software, we did experience some anomalous behavior  when strong convection
was used. In these cases, ad-hoc algebraic multigrid strategies should be adopted.
We also recall that both $\Pi_k$ and $M$ are diagonal,
therefore $L_1$ is obtained from the convection-diffusion matrix by scaling, and
then modifying its diagonal. 

According to \cite{Wathen2008},
we used $A_0=0.9 M$ and $A_1=0.9(\nu\, M)$ for the parameterized preconditioners in (\ref{BT}) within
the {\sc bpcg} iteration. Systems with $L$ to apply $S_0$ in (\ref{BT}) are approximately solved with the 
aforementioned {\sc hsl-mi20} code. 

We set a limit of 80 {\sc gmres} iterations and 1000 {\sc minres} and {\sc bpcg} iterations. If a solver
reaches the maximum number of iterations, the last computed iterate is used as the next
Newton iterate.
 
As for the nonlinear iteration, in all tests
we set the parameter $c$ in the definition of 
the active-set strategy (\ref{AS-def})  equal to one, and  we use
a null starting guess $x_0$ in the Newton iteration,
which by  (\ref{AS-def}) implies that $\mathcal{A}_0=\emptyset$ in all settings. 
As already mentioned, we used the stopping criterion (\ref{crit}) 
with $\eta_k = \eta_k^E$ in (\ref{ex_crit}) where we further included the safeguard $\tau_s=10^{-10}$ as follows
\begin{eqnarray}\label{ex_crit1}
\|J_k x_{k+1}^{j_*} -f_k\|=  \max \{ \tau_s, \eta_k^E\|J_{k}x^0_{k+1}-f_{k}\| \}, 
\end{eqnarray}
$k \ge 1$, with the tight tolerance $\tau_1 = 10^{-10}$ in (\ref{ex_crit}) \cite{ew2}. 
While the residual 2-norm in (\ref{ex_crit1})
 can be cheaply evaluated for {\sc gmres} when using
right preconditioning, in the case
of {\sc minres} we explicitly computed the (unpreconditioned) residual vector at each
iteration, and then computed its norm; for {\sc minres} we thus slightly
modified the code available in \cite{Elman.Ramage.Silvester.07}.

In the numerical tests in Section \ref{exp_inex} 
we also experimented with the adaptive choice $\eta_k=\eta_k^I$ in
(\ref{in_crit}), with $\tau_2 = 10^{-4}, \tau_3 = 10^{-2}$, together with the above safeguard 
threshold $\tau_s$.  We experimentally verified that this choice of tolerances
preserved the global convergence of the active set Newton procedure.

Concerning the outer iteration, we followed \cite{HerzogSachs2010} and we declare convergence
when the nonlinear residual is sufficiently small, i.e.
$$
\|F(u_k,y_k,p_k,\mu_k)\| \le \tau_f, \qquad {\rm with} \quad \tau_f=10^{-8} . 
$$
We verified that this criterion was equivalent
to terminating the iteration  as soon as the active sets stay unchanged in two consecutive steps as proved in \cite{Bergounioux1999,Meyer2007}. On the contrary, any run performing more than 200
nonlinear iterations is considered a failure and will be denoted with the symbol `-' in the forthcoming tables.

All numerical experiments were performed on a  4xAMD Opteron 850, 2.4GHz, 16GB of RAM using 
Matlab R2012a \cite{Matlab}.\\

\subsection{Numerical results}\label{sec:expes} 
The presentation of the numerical results is organized as follows.
Section \ref{exp_comp} is devoted to the comparison of \ascp and \asbd with 
\asbt (see Section \ref{sec:overview} and (\ref{BT})) on symmetric CC problems. Section \ref{exp_new}
collects the numerical results of the new proposals  \ascp and \asbd on symmetric and
nonsymmetric problems for a variety of problem parameters.
Finally, in 
Section \ref{exp_inex} an {\it inexact} active set approach is considered in the solution of 
nonsymmetric CC problems.

In some cases, a comparative computational analysis is carried out by using
performance profiles for a
given set of test problems and a given selection of algorithms \cite{DolanMore02}.
For a problem $P$ in our testing set and an algorithm $A$, we let ${\tt ti}_{P,A}$
denote the total CPU time employed  to solve problem $P$ using
algorithm $A$ and ${\tt ti}_P$ be  the total CPU time 
employed by the {\it fastest} algorithm to solve problem $P$.
As stated in \cite{DolanMore02},
the CPU time performance profile is defined for algorithm $A$ as
\begin{equation*}
\pi_{A}(\tau)
=\frac {\mbox{number of problems s.t. }{\tt ti}_{P,A} \le \tau \, {\tt ti}_P }    
       {\mbox{number of problems } },\;\;\;\; \tau \ge 1 ,
\end{equation*}
that is the probability\footnote{Or, more precisely, the frequency.} 
for solver $A$ that a performance ratio ${\tt ti}_{P,A}/{\tt ti}_P $ is within a factor 
$\tau$ of the best possible ratio. The function $\pi_{A}(\tau)$ is the (cumulative)
distribution function for the performance ratio.

In the upcoming tables of results the following data will be reported:
the average number of linear inner iterations ({\sc li}),
the number of nonlinear outer iterations ({\sc nli} in brackets), the average elapsed CPU time of
the inner solver ({\sc cpu}), and the total elapsed CPU time ({\sc tcpu}).

Finally, to be able to evaluate the effectiveness of the preconditioned linear solvers, we take as
reference the computational cost of solving the whole system with a sparse direct solver
(``backslash'' in Matlab). For the finest mesh, corresponding to $h=2^{-5}$, the 
{\it compiled} direct solver takes 611 seconds to solve a single linear system with $\mathcal{A}_k=\emptyset$ for
some $k$   ($\nu=10^{-2},\beta=0$). We note that this corresponds to
 the cost of the first iteration when the 
active set Newton algorithm is applied to every problem of the family (\ref{pb_gen}).
For comparison purposes, multiplying
by the number of nonlinear iterations, the total cost of the process when
the inner system is solved with a sparse direct method can be derived.


 \begin{table}[htb]
\begin{center}
\scriptsize
\begin{tabular}{|c|c|rrr|rrr|rrr|}
\hline
           &            &    \multicolumn{ 3}{c}{\ascp} &    \multicolumn{ 3}{|c}{ {\asbd} } & \multicolumn{ 3}{|c|}{\asbt} \\

     $\nu$ &        $p$ &   {\sc li} ({\sc nli}) &        {\sc cpu} &       {\sc tcpu} &   {\sc li} ({\sc nli}) &        {\sc cpu} &       {\sc tcpu} &   {\sc li} ({\sc nli}) &        {\sc cpu} &       {\sc tcpu} \\
\hline 
{$10^{-2}$} &   $2$ &    9.6(3) &       0.1 &       0.2 &      20(3) &       0.1 &        0.2 &   11.3(3) &       0.1 &       0.2 \\
{} &   $3$ &    9.5(4) &       0.8 &       3.2 &    19.5(4) &       1.1 &       4.2 &   10.7(4) &       0.7 &       2.7 \\
{} &   $4$ &    8.5(4) &       1.5 &        8.5 &   18.7(4) &       2.5 &       9.9 &   10.0(4) &       6.7 &      26.8 \\
{} &   $5$ &       8.0(4) &      12.1 &      48.2 &   19.2(4) &      36.1 &      144.4 &     9.5(4) &      17.5 &      69.9 \\
\hline
{$10^{-4}$} &   $2$ &    6.5(7) &       0.1 &       0.11 &   13.8(7) &       0.1 &       0.2 &   17.5(7) &       0.2 &       1.3 \\
{} &   $3$ &  11.2(11) &       0.7 &       8.1 &  23.8(11) &       1.3 &      14.4 &  21.1(11) &       1.3 &      14.7 \\
{} &   $4$ &  10.7(17) &       1.8 &      30.1 &  23.5(17) &       3.1 &      51.6 &  18.0(17) &       4.7 &      80.1 \\
{} &   $5$ &  10.3(15) &      16.1 &      241.4 &  24.3(15) &      31.6 &     474.3 &  18.2(15) &      30.9 &     463.3 \\
\hline 
{$10^{-6}$} &   $2$ &   10.3(9) &       0.1 &       0.2 &   22.7(9) &       0.1 &       0.35 &   41.1(9) &       0.1 &       0.7 \\
{} &   $3$ &  16.0(19) &       1.1 &      21.5 &  34.6(19) &       1.9 &      35.8 &     99.0(19) &       6.1 &     115.6 \\
{} &   $4$ &  17.6(54) &       2.9 &      160.7 &  44.9(54) &       5.7 &     289.8 &  93.5(54) &      13.6 &     735.6 \\
{} &   $5$ &  22.0(68) &      38.4 &    2608.4 &  56.3(89) &      63.2 &    5627.2 & 102.1(68) &     136.7 &    9293.6 \\
\hline
{$10^{-8}$}    &   $2$ &   11.1(9) &       0.1 &       0.2 &   25.4(9) &       0.1 &       0.4 &   58.6(9) &       0.1 &       1.0 \\
 &   $3$ &   18.3(27) &       0.7 &      20.2 &  40.1(27) &       2.1 &      57.6 & 133.2(27) &       8.3 &     224.1 \\
 &   $4$ &   30.3(74)	& 7.3	& 540.5 &  72.1(66) &       9.2 &     513.4 & 385.0(66) &      60.1 &    3962.8 \\
 &   $5$ &    - & - & - &    - & - & - &   - & - & - \\
\hline
\end{tabular}  
\caption{Comparison among {\sc as-gmres-ipf}, {\sc as-minres-bdf} and {\sc as-bpcg-bt}. Test problem {\tt CC-Pb1} 
for a variety of {\color{black}$h=2^{-p}$} and $\nu$ ($L$ symmetric, i.e., $\beta=0$).}\label{tab1}
\end{center}
\end{table}

 \begin{table}[htb]
\begin{center}
\scriptsize
\begin{tabular}{|c|c|rrr|rrr|rrr|}
\hline
           &            &    \multicolumn{ 3}{c}{\ascp} &    \multicolumn{ 3}{|c}{ {\asbd} } & \multicolumn{ 3}{|c|}{\asbt} \\

      $\nu$ &        $p$ &   {\sc li} ({\sc nli}) &        {\sc cpu} &       {\sc tcpu} &   {\sc li} ({\sc nli}) &        {\sc cpu} &       {\sc tcpu} &   {\sc li} ({\sc nli}) &        {\sc cpu} &       {\sc tcpu} \\
\hline
{$10^{-2}$} &   $2$ &    8.75(4) &       0.1 &       0.2 &      18(4) &       0.1 &       0.2 &      10.0(4) &       0.03 &       0.10 \\

{} &   $3$ &       8.0(5) &       0.2 &       0.8 &    16.8(5) &       0.9 &       4.7 &       9.0(5) &       0.2 &       0.99 \\

{} &   $4$ &     7.4(5) &       1.3 &       6.5 &    16.2(5) &       2.2 &      11.1 &     9.2(5) &       1.8 &       8.77 \\

{} &   $5$ &     7.4(5) &      11.3 &      56.4 &    16.6(5) &      19.4 &      96.7 &   8.4(5)	& 15.2 &	76.1  \\
\hline
{$10^{-4}$} &   $2$ &   11.1(9) &       0.1 &       0.2 &   23.2(9) &       0.1 &       0.4 &   28.4(7) &       0.1 &       0.6 \\

{} &   $3$ &  12.9(13) &       0.3 &       3.8 &  27.7(13) &       1.5 &      19.8 &  24.4(13) &       0.5 &       6.4 \\

{} &   $4$ &  13.0(14) &       2.1 &      29.5 &  28.7(14) &       3.7 &      52.1 &  20.0(14) &       3.6 &      50.1 \\

{} &   $5$ &  11.7(13) &      18.5 &     240.8 &  27.5(13) &      32.0 &     416.1 &   20.5(13)	& 28.2 &	367.5  \\
\hline
{$10^{-6}$} &   $2$ &  12.2(12) &       0.1 &       0.3 &  26.6(12) &       0.1 &       0.5 &  52.2(12) &       0.1 &       1.3 \\

{} &   $3$ &  16.8(22) &       0.4 &       8.8 &  36.8(22) &       2.0 &      44.6 & 115.5(22) &       2.1 &      46.7 \\

{} &   $4$ &  18.2(35) &       3.1&     106.9 &  43.5(36) &       5.7 &     204.8 & 118.5(35) &      19.3 &     675.2 \\

{} &   $5$ &  20.0(41) &      34.6 &    1416.9 &  52.5(53) &      59.5 &    3151.9 &      84.3(40)	& 109.2	& 4367.1
 \\
\hline
{$10^{-8}$} &   $2$ &  10.4(11) &       0.1 &       0.2 &  23.2(11) &       0.1 &       0.4 &  64.3(11) &       0.1 &       1.6 \\

{} &   $3$ &  15.7(19) &       0.4 &       8.2 &  35.5(19) &       1.9 &      37.1 & 195.1(19) &       3.8 &      71.2 \\

{} &   $4$ &  27.6(55) &       5.3 &     289.1 &     69.0(63) &       9.1 &     572.0 &360.5(54) &      55.9 &    3021.7 \\

{} &   $5$ & 41(156) &      90.7 &      14156.0 &   - & - & -&  343.3(131)	& 438.2 &	57406.9\\
\hline
\end{tabular}  
\caption{Comparison among {\sc as-gmres-ipf}, \asbd and {\sc as-bpcg-bt}. Test with {\tt CC-Pb2} for a variety of 
{\color{black}$h=2^{-p}$} and $\nu$ ($L$ symmetric, i.e., $\beta=0$).}\label{tab1_bis}
\end{center}
\end{table}

\subsubsection{Comparison with the BPCG approach} \label{exp_comp}
In order to make comparisons with \asbt in the setting 
used in \cite{StollWathen2012}, we 
restrict our testing set to symmetric CC problems {\tt CC-Pb1} and {\tt CC-Pb2}
with $\beta=0$.
Numerical results are reported in Tables \ref{tab1} and \ref{tab1_bis}.
The number of nonlinear iterations remains quite low for most choices of the
parameters, except for the finest grid and the limit case $\nu=10^{-8}$.
All methods seem to show some $\nu$-dependence both in the (inner) linear
solver, and in the (outer) nonlinear iteration; however, while in both
problems for \ascp and \asbd such dependence is rather mild, this is significantly
more evident for {\sc as-bpcg-bt}. Large values of {\sc li} for \asbt in the tables
correspond to runs where the maximum number of inner iterations is reached.
This shortcoming makes \asbt not competitive in almost
all parameter combinations, with timings that differ significantly from the
other methods, up to at most one order of magnitude.
Finally, we recall that at each iteration \ascp and \asbd solve linear systems 
of dimension $3 n_h + n_{{\mathcal{A}}_k}$, whereas
\asbt solves systems of fixed dimension $3 n_h$. 
The numbers in Tables \ref{tab1} and \ref{tab1_bis} 
 show that an appropriate explicit 
treatment of the active-set information within the preconditioner is capable of
making up for the larger problem size, yielding an overall significant gain in CPU time.

\begin{remark}
{\rm\color{black}
For the sake of completeness, we also solved the problem with the block triangular
preconditioner suggested in \cite{StollWathen2012}, with {\sc gmres} as a solver
instead of {\sc bpcg}. Results are reported in Table \ref{tab1_2_gmres} for a
selection of parameters and for both problems. The results do not differ from
those showed in the previous tables, indicating that the chosen linear solver is
not responsible for the unsatisfactory performance of the preconditioned iteration.
Because of the use of {\sc gmres}, memory requirements are clearly superior to those
for {\sc bpcg}.

\begin{table}[htb]
{\color{black}
\begin{center}
\footnotesize
\begin{tabular}{|c|c|rrr|rrr|}
\hline
           &            &    \multicolumn{6}{c|}{ \sc as-gmres-bt } \\
\hline
           &            &    \multicolumn{ 3}{c|}{\tt CC-Pb1} &    \multicolumn{ 3}{c|}{\tt CC-Pb2} \\           
      $\nu$ &        $p$ &   {\sc li} ({\sc nli}) &        {\sc cpu} &       {\sc tcpu} &   {\sc li} ({\sc nli}) &        {\sc cpu} &       {\sc tcpu} \\
\hline
{$10^{-2}$} &   $2$ & 10.7(4) & 0.1  & 0.4  & 9.25(4) & 0.06 & 0.3 \\

{} &   $3$ &  9.8(4)  &  1.3  & 5.1  &  7.8(5) & 0.3  &  1.6  \\

{} &   $4$ &  8.5(4)  &  7.9  &  31.9 &  7.6(5)  & 2.4  & 12.2   \\

{} &   $5$ & 8.5(4)  & 19.5 & 77.9  &  7.2(5) &  17.5  & 87.7  \\
\hline
{$10^{-4}$} &   $2$ & 14.1(7)  &  0.3  & 2.2  &  22.0(9)  &  0.1   &  0.8  \\

{} &   $3$ &  16.5(11) & 1.8 &  20.2  & 19.9(13) & 0.7 & 8.8 \\

{} &   $4$ &  14.2(17) &  4.2  &  71.5  &  16.5(14) &  4.2  & 59.0  \\

{} &   $5$ &  13.9(15) &   30.2   &  452.5   &  15.8(13) &  33.4  & 433.6 \\
\hline
{$10^{-6}$} &   $2$ &  25.1(9)  &  0.1   &  0.9  & 31.3(12)  &  0.1  & 1.7  \\

{} &   $3$ & 50.6(19)  & 6.6  & 125.7  & 55.6(22) & 3.5  & 77.6  \\

{} &   $4$ & 46.9(54) & 15.0  & 810.9  & 55.3(35) & 19.8 & 691.2 \\

{} &   $5$ & 49.7(68) & 141.5  & 9621.7 & 41.1(40)  &  117.5   &  4700.1  \\
\hline
{$10^{-8}$} &   $2$ & 27.9(9)  &  0.1  & 1.1  & 37.2(11)  &  0.2  &  2.1  \\

{} &   $3$ & 56.9(27)  & 8.6 &  232.0  &  75.4(19) &  7.6  &  144.7 \\

{} &   $4$ & 115.5(74)  &  56.7  &  4134.5   & 164.8(54)  &  109.9  & 5932.9 \\

{} &   $5$ & - & -  &  -  & 102.1(120)  & 507.7 & 60927.8 \\
\hline
\end{tabular}  
\caption{\color{black}Performance results for {\sc as-gmres-bt}. Test with {\tt CC-Pb1} and {\tt CC-Pb2} for a variety of 
{\color{black}$h=2^{-p}$} and $\nu$ ($L$ symmetric, i.e., $\beta=0$).\label{tab1_2_gmres}}
\end{center}
}
\end{table}


}
\end{remark}

\subsubsection{Dependence on the problem parameters} \label{exp_new}
We tested the new preconditioners on CC, MC and SC problems  by
analyzing their dependence on the parameters of the discretized problem, i.e. the regularization parameter $\nu$,
the convection coefficient $\beta$, the mesh size $h$
and, for the MC case, the regularization parameter $\epsilon$.

 \begin{table}[htb]
\begin{center}
\footnotesize
\begin{tabular}{|c|c|rr|rr|rr|rr|}
\hline
                    &            & \multicolumn{ 2}{c|}{$\nu=10^{-2}$} & \multicolumn{ 2}{c|}{$\nu=10^{-4}$} & \multicolumn{ 2}{c|}{$\nu=10^{-6}$} & \multicolumn{ 2}{c|}{$\nu=10^{-8}$} \\

   $\beta_1$ &        $p$ &    {\sc li}({\sc nli}) &       {\sc tcpu} &    {\sc li}({\sc nli}) &       {\sc tcpu} &    {\sc li}({\sc nli}) &       {\sc tcpu} &    {\sc li}({\sc nli}) &       {\sc tcpu} \\
\hline
     0 &       $2$ &    9.6(3) &       0.2 &    6.5(7) &       0.1 &   10.3(9) &       0.2 &   11.1(9) &       0.2 \\

           &   $3$ &    9.5(4) &       3.2 &  11.2(11) &       8.0 &  16.0(19) &      21.5 &   18.3(27) &      20.2 \\

           &   $4$ &    8.5(4) &        8.5 &  10.7(17) &      30.1 &  17.6(54) &      160.7 &  30.3(74) &     540.5 \\

           &   $5$ &       8.0(4) &      48.2 &  10.3(15) &      241.4 &  22.0(68) &    2608.4 & - & -  \\
\hline
        10 &   $2$ &       9.0(3) &       0.1&    8.3(10) &       0.2 &   10.4(10) &       0.3 &  11.3(10) &       0.3 \\

           &   $3$ &     8.5(4) &       0.7 &  10.5(13) &          3.0 &  15.4(18) &        6.8 &  19.8(19) &      10.7 \\

           &   $4$ &     8.5(4) &        6.1 &  10.8(13) &      25.3 &  18.6(41) &     135.9 & 23.8(109) &     509.1 \\

           &   $5$ &       8.0(4) &      53.6 &     11.0(15) &     277.6 &  20.9(47) &     1810.7* & 36.9(164) &  13203.7* \\
\hline
       $10^2$ &   $2$ &       5.0(3) &       0.1 &       7.0(4) &       0.1 &      10.0(6) &       0.1 &   13.7(8) &       0.3 \\

           &   $3$ &       6.0(3) &       0.4 &     9.6(5) &        0.9 &  12.3(12) &       3.5 &  23.7(19) &       12.9 \\

           &   $4$ &    5.3(3) &       2.9 &    8.8(6) &       8.9 &  15.1(14) &      41.4 &  34.3(46) &     337.9 \\

           &   $5$ &    7.3(3) &      40.3 &      10.0(6) &      108.2 & 14.4(19) &  690.5* &  40.0(81) &   8383.7* \\
\hline
      $10^3$ &   $2$ &       3.0(2) &       0.1 &     4.5(2) &       0.1 &       6.0(4) &       0.1 &    8.8(6) &       0.2 \\

           &   $3$ &       4.0(2) &       0.2 &       5.0(2) &       0.2 &    5.8(6) &       0.8 &  16.3(18) &       7.6 \\

           &   $4$ &     4.5(2) &       1.7 &     6.5(2) &       2.4 &    8.1(6) &       9.1 &  18.0(14) &      53.2 \\

           &   $5$ &     4.5(2) &      29.3 &    5.6(3) &      53.5 &    7.2(7) &      156.5 & 25.0(26) &    2517.2* \\      
\hline
\multicolumn{10}{c}{   } \\
\hline
                    &            & \multicolumn{ 2}{c|}{$\nu=10^{-2}$} & \multicolumn{ 2}{c|}{$\nu=10^{-4}$} & \multicolumn{ 2}{c|}{$\nu=10^{-6}$} & \multicolumn{ 2}{c|}{$\nu=10^{-8}$} \\

   $\beta_1$ &        $p$ &    {\sc li}({\sc nli}) &       {\sc tcpu} &    {\sc li}({\sc nli}) &       {\sc tcpu} &    {\sc li}({\sc nli}) &       {\sc tcpu} &    {\sc li}({\sc nli}) &       {\sc tcpu} \\
\hline
      0 &      $2$ &    8.7(4) &       0.2 &   11.1(9) &       0.2 &  12.1(12) &       0.4 &  10.4(11) &       0.3 \\

           &   $3$ &       8.0(5) &       0.8 &  12.9(13) &       3.8 &  16.8(22) &       8.8 &  15.7(19) &       8.2 \\

           &   $4$ &     7.4(5) &       6.5 &  13.0(14) &      29.5 &  18.2(35) &     106.9 &  27.6(55) &     289.1 \\

           &   $5$ &     7.4(5) &      56.4 &  11.7(13) &     240.8 &  20.0(41) &    1416.9 & 41.0(156) &      14156.0 \\
\hline
        10 &   $2$ &       8.0(4) &       0.1 &   10.6(10) &       0.3 &   13.8(15) &       0.5 &  15.1(15) &       0.6 \\

           &   $3$ &       8.0(4) &       0.7 &  13.1(12) &       3.5 &  19.3(31) &      14.7&  24.6(30) &      21.5 \\

           &   $4$ &     6.4(5) &       6.7 &   13.5(12) &      28.7 &  20.6(46) &      167.6 &  34.1(67) &     449.5 \\

           &   $5$ &     6.6(5) &      59.1 &  12.3(13) &     273.1 &  21.7(58) &    2291.8 & 41.4(162) &  15118.4*  \\
\hline
       $10^2$ &   $2$ &     4.5(2) &       0.1 &    9.8(6) &       0.2 &  12.3(10) &       0.3 &   15.5(12) &       0.5 \\

           &   $3$ &    4.3(2) &       0.3 &   10.3(6) &       1.1 &  15.7(16) &       5.2 &  27.2(24) &       18.5 \\

           &   $4$ &    4.6(3) &       2.7 &    9.6(6) &       9.7 &  18.4(20) &      61.5 &  33.2(47) &     321.4 \\

           &   $5$ &    5.6(3) &      69.3 &    8.5(7) &     196.7 &  17.6(23) &     907.8 &  34.3(82) &    7428.8 \\
\hline
      $10^3$ &   $2$ &       3.0(2) &       0.1 &     5.0(3) &       0.1 &   10.1(7) &       0.2 &   13.6(9) &       0.3 \\

           &   $3$ &     2.5(2) &       0.1 &       5.0(3) &       0.3 &   10.1(7) &       1.4 &  22.3(12) &       6.8 \\

           &   $4$ &     2.5(2) &       0.2 &     4.5(4) &       3.7&    9.5(7) &       12.1 &  21.1(15) &      58.9 \\

           &   $5$ &       3.0(2) &       25.9 &    4.2(4) &      66.4 &    8.2(8) &     237.4 &  19.3(17) &  2277.6*  \\      
\hline
\end{tabular}  
\caption{\ascp for a variety of values for {\color{black}$h=2^{-p}$, $\nu$ and $\beta=(\beta_1,0,0)$}.
The symbol `*' denotes runs where an  {\sc hsl-mi20} warning occurred; in some of these cases, much larger timings were observed.
Top: {\tt CC-Pb1}. Bottom: {\tt CC-Pb2}.
} \label{tabcpf_pb2}
\end{center}
\end{table}

 \begin{table}[htb]
\centering
\scriptsize
\begin{tabular}{|c|c|rr|rr|rr|rr|}
\hline
                    & $p$        & \multicolumn{ 2}{c|}{$\nu=10^{-2}$} & \multicolumn{ 2}{c|}{$\nu=10^{-4}$} & \multicolumn{ 2}{c|}{$\nu=10^{-6}$} & \multicolumn{ 2}{c|}{$\nu=10^{-8}$} \\

   $\beta_1$ &       &    {\sc li}({\sc nli}) &       {\sc tcpu} &    {\sc li}({\sc nli}) &       {\sc tcpu} &    {\sc li}({\sc nli}) &       {\sc tcpu} &    {\sc li}({\sc nli}) &       {\sc tcpu} \\
\hline
 
         0 &   $2$ &      20.0(3) &        0.2 &   13.8(7) &        0.2 &   22.7(9) &        0.4 &   25.4(9) &        0.4 \\

           &   $3$ &    19.5(4) &        4.2 &  23.8(11) &       14.4 &  34.6(19) &       35.8 &  40.1(27) &       57.6 \\

           &   $4$ &   18.7(4) &        9.9 &  23.5(17) &       51.6 &  44.9(54) &      289.8 &  72.1(66) &      513.4 \\

           &   $5$ &   19.2(4) &      144.4 &  24.3(15) &      474.3 &  56.3(89) &     5627.2 &          - &          - \\
\hline
        10 &   $2$ &   18.3(3) &        0.1 &  18.3(10) &        0.3 &  25.3(10) &        0.5 &     29.0(10) &        0.5 \\

           &   $3$ &   17.7(4) &        4.2 &  24.6(13) &       19.0 &  37.7(18) &       39.7 &  50.6(19) &       56.7 \\

           &   $4$ &   17.7(4) &       10.8 &  26.5(13) &       40.0 &  53.7(33) &      245.7 &  87.7(42) &      500.7 \\

           &   $5$ &   19.2(4) &       98.4 &  29.5(15) &      550.6 &  63.1(74) &     5572.9 & 174.1(146)$\dagger$	& $>5h$*$\dagger$  \\ 
\hline
       $10^2$ &   $2$ &   10.5(2) &        0.1 &      14.0(4) &        0.1 &   20.5(6) &        0.2 &    31.5(8) &        0.4 \\

           &   $3$ &   11.6(3) &        1.8 &   20.2(5) &        5.2 &  27.1(12) &       17.4 &  53.5(19) &       54.3 \\

           &   $4$ &   11.6(3) &        5.3 &   20.5(6) &       17.8 &     37.0(14) &       72.8 &  93.5(40) &      546.5 \\

           &   $5$ &    98.3(3)	& 737.3 &	28.5(6) &	382.7	& 74.8(19)$\dagger$	& 2988.0*$\dagger$ &180.2(80)$\dagger$	& $>5h$*$\dagger$\\ 
\hline
      $10^3$ &   $2$ &     6.5(2) &        0.1 &     8.5(2) &        0.1 &    11.5(4) &        0.1 &    17.5(6) &        0.2 \\

           &   $3$ &     7.5(2) &        0.9 &    10.5(2) &        1.1 &   11.8(6) &        4.0 &      29.0(8) &       13.5 \\

           &   $4$ &     9.5(2) &        3.1 &    13.5(2) &        4.3 &   16.8(6) &       15.6 &  41.1(10) &       61.7 \\

           &   $5$ &    9.5(2)	& 79.4 &	11.6(3) &	146.3 &	19.0(7) &	549.3 &	42.1(16) &	3036.1* \\
\hline
%
%
\multicolumn{10}{c}{   } \\
\hline
                    &    $p$  & \multicolumn{ 2}{c|}{$\nu=10^{-2}$} & \multicolumn{ 2}{c|}{$\nu=10^{-4}$} & \multicolumn{ 2}{c|}{$\nu=10^{-6}$} & \multicolumn{ 2}{c|}{$\nu=10^{-8}$} \\

   $\beta_1$ &       &    {\sc li}({\sc nli}) &       {\sc tcpu} &    {\sc li}({\sc nli}) &       {\sc tcpu} &    {\sc li}({\sc nli}) &       {\sc tcpu} &    {\sc li}({\sc nli}) &       {\sc tcpu} \\
\hline
0 &   $2$ &      18.0(4) &        0.2 &   23.2(9) &        0.4 &  26.5(12) &        0.6 &  23.1(11) &        0.5 \\
           &   $3$ &    16.8(5) &        4.7 &  27.7(13) &       19.8 &  36.8(22) &       44.6 &  35.5(19) &       37.1 \\
           &   $4$ &    16.2(5) &       11.1 &  28.7(14) &       52.1 &  43.5(36) &      204.8 &     69.0(63) &      572.0 \\
           &   $5$ &    16.6(5) &       96.7 &  27.5(13) &      416.1 &  52.5(53) &     3152.0 & 123(133)$\dagger$	&$ >5h\dagger$	\\ 
\hline
        10 &   $2$ &   16.7(4) &        0.1 &   23.0(10) &        0.4 &  32.4(15) &        0.9 &  37.4(15) &        1.0 \\
           &   $3$ &    16.5(4) &        4.0 &  30.6(12) &       21.7 &  52.2(30) &       90.8 &  70.3(30) &      122.1 \\
           &   $4$ &   14.2(5) &       10.8 &     32.0(12) &       54.6 &  63.7(47) &      409.8 & 108(79) &     1151.1 \\
           &   $5$ &   14.8(5) &       98.8 &     33.0(13) &      533.3 & 78.4(100) &    $>5h$&  194(159)$\dagger$ & $>5h$*$\dagger$ \\ 
\hline
       $10^2$ &   $2$ &     9.5(3) &        0.1 &   20.6(6) &        0.2 &  27.1(10) &        0.4 &  35.6(12) &        0.7 \\
           &   $3$ &       9.0(3) &        1.5 &      22.0(6) &        7.0 &  37.3(16) &       31.5 &  66.2(25) &       86.9 \\
           &   $4$ &    9.6(3) &        4.7 &   22.1(6) &       19.1 &  47.1(20) &      138.3 &  90.9(56) &      744.5 \\
           &   $5$ &  226(3) &	1206.2 &112.5(7) &	1411.8	& 109.2(23)$\dagger$ &	4762.4$\dagger$ & -	& - \\
\hline
      $10^3$ &   $2$ &     5.5(2) &        0.1 &   10.3(3) &        0.1 &      21.0(7) &        0.3 &   28.4(9) &        0.5 \\
           &   $3$ &     5.5(2) &        0.6 &   10.3(3) &        1.8 &   20.4(7) &        7.7 &  49.6(11) &       31.3 \\
           &   $4$ &     5.5(2) &        2.3 &      10.0(4) &        6.6 &   19.5(7) &       21.2 &    63.8(15)$\dagger$	&     699.3$\dagger$	\\
           &   $5$ &   6.5(2)	& 71.1 &	8.0(5)	& 8.3& 25.5(8)$\dagger$	& 828.3$\dagger$ &	67.7(26)$\dagger$	& 7897.7$\dagger$	 \\
\hline
\end{tabular}  
 \caption{\asbd for a variety of values for {\color{black}$h=2^{-p}$, $\nu$ and $\beta=(\beta_1,0,0)$}.
The symbol `*' denotes runs where an MI20 warning occurred; in some of these cases, much larger timings were observed 
($>5h$ means that the CPU is larger  than 5 hours).
 Top: {\tt CC-Pb1}. Bottom: {\tt CC-Pb2}.  \label{tabbdf_pb2} }
\end{table}

\vskip 0.04in
\begin{example}\label{ex:1}
{\rm
For the CC problems, we varied
$h \in \{2^{-2},2^{-3},2^{-4},10^{-5}\},\quad \nu \in \{10^{-2},10^{-4},10^{-6},10^{-8}\},$ and
we set
$\beta = \left (\beta_1,0,0 \right)$ with $\beta_1\in \{0,10,100,1000\}$.
We remark that $\nu=10^{-8}$ was included for completeness, however it will be considered
as a limit case because it is rather small. Analogously, $\beta_1=1000$ makes the
operator very convection-dominated, providing anomalous behaviors in some exceptional
cases; we did not explore whether for this extreme value of $\beta_1$
the upwind discretization was sufficient in these cases to
damp the well-known numerical instabilities arising in the discretization phase.  In fact, the
value $\beta_1=1000$ was only considered for consistency with respect to
the experiments carried out in \cite{HerzogSachs2010}. In the same lines, we prefer to limit
our speculations on the dependence with respect to $\beta$ to the empirical level, as a deeper analysis
would require a thorough discussion of both the discretization strategy and the employed convection;
this is clearly beyond the scope of this paper.

We collect the results obtained with  \ascp  and \asbd for the problems {\tt CC-Pb1} and {\tt CC-Pb2} 
in Tables  \ref{tabcpf_pb2}-\ref{tabbdf_pb2} and the corresponding total CPU time performance
profile is displayed in Figure \ref{perfprof_cpfbdf} (left plot) varying all the
parameters for a total of 128 runs.  The average  number of inner iterations is quite
homogeneous with respect to $h$ and slightly dependent on $\nu$ and $\beta$.
A comparison of Tables \ref{tabcpf_pb2} and \ref{tabbdf_pb2} shows that 
the number of nonlinear iterations is quite different between \ascp and \asbd
when $h$ is small and $\nu\in \{10^{-6},10^{-8}\}$. For these values the preconditioner
in \asbd is rather ill-conditioned and its performance deteriorates. In this case, the Newton steps computed 
with the two preconditioned solvers, using the stopping criterion
(\ref{ex_crit1}), might differ so greatly that different convergence histories take place.  
Unfortunately, this resulted in the \asbd failure  in 10 instances. We recovered 9 over 10 failures by imposing
the stricter tolerances
$\tau_s=\tau_1=10^{-12}$ in (\ref{ex_crit1}) (this runs are marked with the symbol `$\dagger$' in 
Table \ref{tabbdf_pb2}). A few unexpected large 
values of {\sc li} can still be observed in Table~\ref{tabbdf_pb2} for $\beta=100$ and $h=2^ {-5}$,
which can be presumably ascribed to an inaccuracy of the multigrid operator. 

The superiority of \ascp is also evident in the left plot of Figure \ref{perfprof_cpfbdf}, which
reveals that \ascp is much more efficient than \asbd in terms of total CPU time
and that in the 55\% of the runs, the CPU time employed by \asbd is  within
a factor 2 of the time employed by {\sc as-gmres-ipf}.

Finally, for the sake of completeness, we also carried out experiments on {\tt CC-pb1}
using the {\sc agmg} algebraic multigrid operator \cite{Notay2012,AGMG} in place of the 
{\sc hsl-mi20} in the solution of systems with $L_1$. The implementation of 
{\sc agmg} requires the use of the ``flexible'' variant of the linear system solver since the application of
multigrid preconditioner is the result of an iterative process and therefore it changes step by step \cite{Saad-book}.
Table \ref{tab_agmg} shows the results obtained using Flexible {\sc gmres} ({\sc fgmres}) in combination with
{\sc hsl-mi20} (first two columns) and {\sc agmg} (last two columns) in the application of the 
$\pc$ preconditioner. We only report experiments with 
$\nu\in \{10^{-6},10^{-8}\}$, as for larger values the performance  with the two 
multigrid preconditioners is very similar.  For $\nu=10^{-6}$ the overall performance in terms of
CPU time is still somewhat comparable, whereas it is clearly in favor of {\sc agmg} in the extreme
case $\nu=10^{-8}$. On the other hand, the average number of iterations ({\sc li}) with {\sc hsl-mi20} is in
general lower, showing that the latter preconditioner is more effective in terms of approximation
properties, but more expensive to apply.
}
\end{example}

 \begin{table}[htb]
\begin{center}
\footnotesize
\begin{tabular}{|r|c|rr|rr|rr|rr|}
\hline
    &            & \multicolumn{ 4}{c|}{ {\sc as-fgmres-ipf} with {\sc hsl-mi20}} & \multicolumn{ 4}{c|}{{\sc as-fgmres-ipf} with {\sc agmg}} \\
    \hline
                    &            & \multicolumn{ 2}{c|}{$\nu=10^{-6}$} & \multicolumn{ 2}{c|}{$\nu=10^{-8}$} & \multicolumn{ 2}{c|}{$\nu=10^{-6}$} & \multicolumn{ 2}{c|}{$\nu=10^{-8}$} \\

   $\beta_1$ &        $p$ &    {\sc li}({\sc nli}) &       {\sc tcpu} &    {\sc li}({\sc nli}) &       {\sc tcpu} &    {\sc li}({\sc nli}) &       {\sc tcpu} &    {\sc li}({\sc nli}) &       {\sc tcpu} \\
      \hline
         0 &   $2$ &   10.3(9) &       0.3 &   11.1(9) &       0.4 &      17.0(9) &       0.3 &   18.7(9) &       0.3 \\
           &   $3$ &  16.0(19) &      21.3 &   18.3(27) &      36.4 &  23.7(19) &      12.3 &  28.5(27) &      27.4 \\
           &   $4$ &  17.6(54) &     223.6 &  28.5(57) &     351.6 &  28.5(54) &     157.4 &  41.7(57) &      286.7 \\
           &   $5$ &  21.9(68) &    2794 & 38.7(190) &    17457.3 &  44.8(68) &     3766.9 &  53.3(189) &    14250 \\
\hline
        10 &   $2$ &   10.4(10) &       0.2 &  11.3(10) &       0.4 &     19.0(10) &       0.2 &   20.7(10) &       0.3 \\
           &   $3$ &  15.4(18) &      20.9 &  19.7(19) &      31.1 &     25.0(18) &      11.5 &  31.7(19) &      23.1\\
           &   $4$ &  18.6(41) &      132.6 &  26.4(42) &     239.4 &  25.0(41) &      79.1 &  35.4(42) &     272.3 \\
           &   $5$ &  20.9(47) &    1938 & 37.9(138) &  12582* &   39.3(47) &    1951.3 & 50.2(147) &    10105 \\
\hline
    $10^2$ &   $2$ &      10.0(6) &       0.1 &   13.7(8) &       0.5 &      16.0(6) &       0.1 &   20.3(8) &       0.3 \\
           &   $3$ &  12.3(12) &      10.5 &  23.7(19) &      36.5 &   20.5(12) &       7.5 &  31.5(19) &       24.4 \\
           &   $4$ &  15.1(14) &      39.8 &  36.8(34) &        321.0 &   28.7(14) &      37.6 &  39.4(34) &     153.3 \\
           &   $5$ & 14.4(19) &    776* &  38.8(82) &    9172* &   39.2(19) &      755.6 &  50.2(81) &    5658 \\
\hline
    $10^3$ &   $2$ &       6.0(4) &       0.1 &    8.8(6) &       0.37 &    8.2(4) &       0.1 &   12.3(6) &       0.2 \\
           &   $3$ &    5.8(6) &       2.1 &   14.1(8) &       9.3 &    9.8(6) &       1.2 &   18.7(8) &       5.1 \\
           &   $4$ &    8.1(6) &       8.5 &  20.1(10) &      43.2 &      12.0(6) &       5.1 &   26.3(10) &      31.6 \\
           &   $5$ &    7.2(7) &     150.1 &  19.6(16) &   1660* &   13.1(7) &      43.9 &  26.6(16) &     600.6 \\
\hline
\end{tabular}  
\caption{{\sc as-fgmres-ipf} (flexible variant) using {\sc hsl-mi20} 
(left) and {\sc agmg} (right) for a variety of values of $h$ and $\beta$, and small values of $\nu$.
The symbol `*' denotes runs where an {\sc hsl-mi20} warning occurred; Test problem {\tt CC-Pb1}. 
} \label{tab_agmg}
\end{center}
\end{table}

\begin{example}\label{ex:beta}
{\rm
 We further investigate the reliability of our proposals considering problem {\tt CC-pb2} with 
 the following nonconstant convection parameter 
\begin{equation}\label{betan}
\beta(x,y,z)  =  \left ( \begin{array}{c}
  - 2x(1-x)(2y-1)z \\
  (2x-1)y(1-y) \\
    (2x-1)(2y-1)z(1-z)  \end{array} \right) ;
\end{equation}
see example 3D1 in \cite{Notay2012}. The performance of \ascp and \asbd
is analogous to that showed in Tables~\ref{tabcpf_pb2}-\ref{tabbdf_pb2} for the constant and 
unidirectional $\beta=(\beta_1,0,0)$; a sample of this behavior for \ascp  is reported in Table~\ref{tab_notay}
as $\nu$ and {\color{black}$h-2^{-p}$} vary.
}
\end{example}

 \begin{table}[htb]
\begin{center}
\footnotesize
\begin{tabular}{|c|rr|rr|rr|rr|}
\hline
          & \multicolumn{ 8}{c|}{ {\sc as-gmres-ipf} on {\tt CC-Pb2} with convection (\ref{betan})} \\
    \hline
                          & \multicolumn{ 2}{c|}{$\nu=10^{-2}$} & \multicolumn{ 2}{c|}{$\nu=10^{-4}$} & \multicolumn{ 2}{c|}{$\nu=10^{-6}$} & \multicolumn{ 2}{c|}{$\nu=10^{-8}$} \\

        $p$ &    {\sc li}({\sc nli}) &       {\sc tcpu} &    {\sc li}({\sc nli}) &       {\sc tcpu} &    {\sc li}({\sc nli}) &       {\sc tcpu} &    {\sc li}({\sc nli}) &       {\sc tcpu} \\
\hline
  $2$ &     5.0(3) &        0.3 &     9.0(8) &        0.4 &   12.4(19) &        0.5 &  14.5(19) &        0.7 \\
  $3$ &     5.0(3) &        2.2 &     9.4(8) &        5.4 &   15.1(32) &       34.2 &   20.1(41) &       66.9 \\
  $4$ &     4.7(3) &        5.9 &     8.3(9) &       14.6 &   15.4(39) &      133.2 &   26.5(92) &      528.8 \\
  $5$ &     5.0(3) &       42.2 &     7.8(9) &      139.9 &   14.9(36)         &  1089.0   &   28.6(137)         &    9643.7   \\
\hline
\end{tabular}
\caption{{\sc as-gmres-ipf} on problem {\tt CC-Pb2} with convection $\beta$ given in (\ref{betan}). 
} \label{tab_notay}
\end{center}
\end{table}

\begin{example}\label{ex:2}
{\rm
For the MC and SC problems, we considered
$h \in \{2^{-2},2^{-3},2^{-4}\}$, $\nu \in \{10^{-2},10^{-4},10^{-6},10^{-8}\}$,
$\beta = \left (\beta_1,0,0 \right)$ with $\beta_1\in \{0,10,100,1000\}$,
and 
$\epsilon \in \{10^{-1},10^{-2},10^{-3},10^{-4},10^{-8},0 \}$,
where the values $\epsilon \in \{10^{-8},0 \}$ are included to comprise 
the SC problems. We thus obtained a set of 288 runs.
The numerical results for these problems do not significant differ from those of the
CC problem, at least for the larger values of $\epsilon$ in the set. 
Therefore, to avoid proliferation of tables, we prefer not to include them, and
report instead the overall performance profile in the right plot of
Figure~\ref{perfprof_cpfbdf}.
For all considered runs,
the profile clearly shows that \ascp is the fastest in the 96\% of the runs 
and that \asbd is within a factor 2 of \ascp for the majority (93\%) of the runs.

\begin{table}[htb]
\scriptsize
\begin{center}
\begin{tabular}{c|cccccc c c|cccccc}
\multicolumn{ 7}{c}{ $\beta_1=10$} &  & \multicolumn{ 7}{c}{ $\beta_1=100$} \\
  $\epsilon$         &  -1 &  -2 &  -3 &  -4 &    -8 &                          $-\infty$&            &     $\epsilon$         &  -1 &  -2 &  -3 &  -4 &    -8 &  $-\infty$   \\
  $\nu$         &     &     &     &     &       &                      &        &         $\nu$         &   &   &   &   &     &     \\
	 \cmidrule{1-7} 	 \cmidrule{9-15}  
        -2 &         \cellcolor{black!15}{10.3}&         14.3 &       35.3 &        32.5 &  34.1       &  34.1    &               &         -2 &  \cellcolor{black!15}{6.0} &         7.7&       8.7 &         9.0 &        9.7 &            9.0 \\
        -4 &         13.5 &        \cellcolor{black!15}{13.3} &         16.6 &       20.2 &         21.0 &    21.3     &                &         -4 &         12.3 &   \cellcolor{black!15}{13.3} & 16.3   &        22.3 &         26.6 &                26.4 \\

        -6 &         19.5 &         16.0 &  \cellcolor{black!15}{14}&    13.5 &         13.5 &         13.5 &                    &         -6 &         21.6 &   19.8 &        \cellcolor{black!15}{14.7} &         17.7 &      16.7 &    16.7 \\
        -8 &         25.8 &         18.4&         12.0 & \cellcolor{black!15}{10.5}&         10.5 &         10.5 &                  &         -8 &         40.4 &         34.2 &         18.0 &      \cellcolor{black!15}{14.0} &         13.5 &  13.5  \\
\end{tabular}  
\end{center}
\caption{Mixed Constraints {\tt MC-Pb1}: Average number of {\sc gmres} iterations using \ascp
  with $h=2^{-4}$ and varying $\nu$ and $\epsilon$ ($\log_{10}$ values of $\nu,\epsilon$).} \label{MC_eps_nu}
\vskip 5pt	
\scriptsize
\begin{center}
\begin{tabular}{c|cccccc c c|cccccc}
\multicolumn{ 7}{c}{ $\beta_1=10$} & & \multicolumn{ 7}{c}{ $\beta_1=100$} \\
  $\epsilon$         &  -1 &  -2 &  -3 &  -4 &    -8 &                          $-\infty$&        &         $\epsilon$         &  -1 &  -2 &  -3 &  -4 &    -8 &  $-\infty$   \\
  $\nu$         &     &     &     &     &       &                      &         &     $\nu$         &   &   &   &   &     &     \\
	 \cmidrule{1-7} 	 \cmidrule{9-15}  
        -2 &         \cellcolor{black!15}{22.0}&        32.1 &       60.7 &        86.2 &    91*     &  93*    &               &         -2 &  \cellcolor{black!15}{14.0} &         17.3&       19.2 &   20.0 & 22.6&  21.3 \\
        -4 &         27.8 &        \cellcolor{black!15}{27.0} &         34.4 &       42.2 &      44.1&        44.8   &                &         -4 &        27.0 &   \cellcolor{black!15}{28.0} & 33.7   &        44.2 &    56.0 &   55.2 \\
        -6 &         45.3 &        33.2 &  \cellcolor{black!15}{27.5}&    27.5 &       21.5 &      27.5 &       &                     -6 &        49.4 &   41.4 &        \cellcolor{black!15}{29.7} &         36.0 &   34.7 &   34.7 \\
        -8 &         65.7 &         39.8&        24.5 & \cellcolor{black!15}{21.5}&     21.5 &    21.5 &      &                    -8 &         93.6 &         71.9 &       36.3 &      \cellcolor{black!15}{28.5} &      27.5 &  27.5 \\
\\
        \end{tabular}  
\end{center}
{\footnotesize * 6 pre/post smoothing steps set in {\sc hsl-mi20}}
 \caption{Mixed Constraints {\tt MC-Pb1}: Average number of {\sc minres} iterations using \asbd
  with $h=2^{-4}$ and varying $\nu$ and $\epsilon$ ($\log_{10}$ values of $\nu,\epsilon$).} \label{MC_eps_nu_bdf}
\end{table}

 \begin{figure}[htb]
   \centering
       \includegraphics[width=2.5in,height=2.2in]{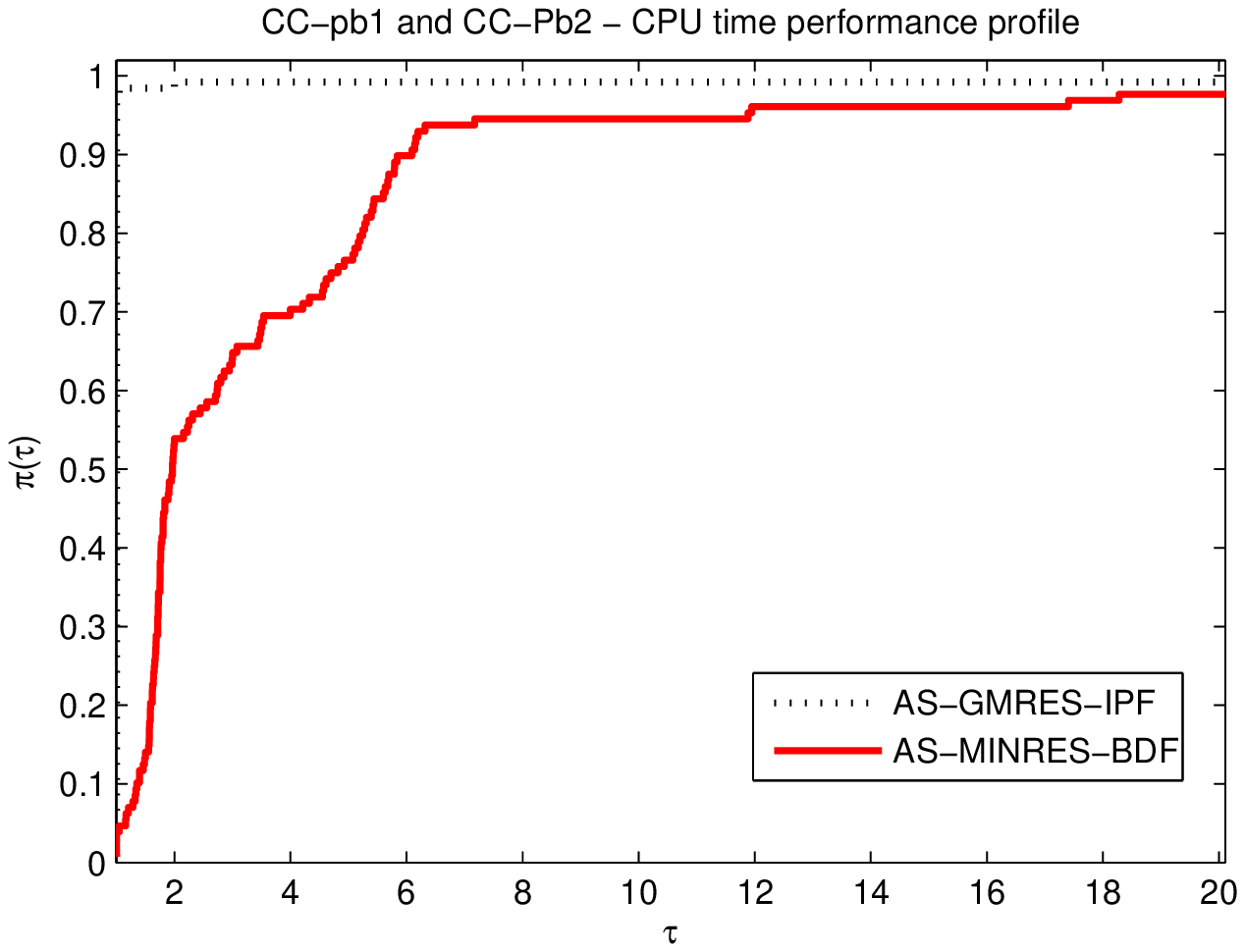}    
       \includegraphics[width=2.5in,height=2.2in]{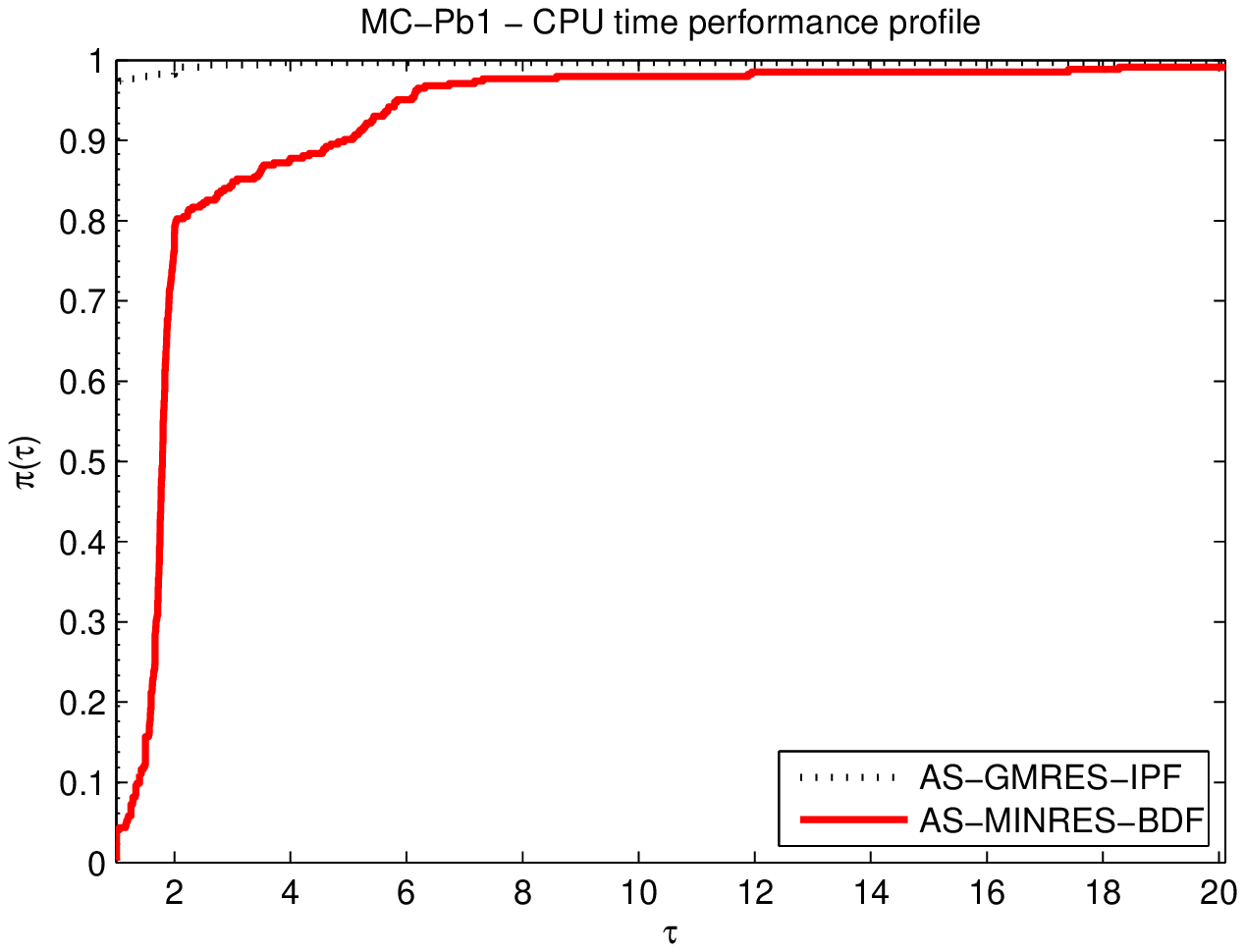}  
 \caption{Total CPU time performance profile for \ascp and {\sc as-minres-bdf}. Left: {\tt CC-Pb1} and {\tt CC-Pb2}.
 Right: {\tt MC-Pb1}. \label{perfprof_cpfbdf}}
 \end{figure}

{\color{black}
The dependence on $\epsilon$ and the mutual influence of $\epsilon$ and $\nu$
deserve deeper exploration.}
In Tables \ref{MC_eps_nu} and \ref{MC_eps_nu_bdf} 
we report the average number of inner iterations for 
$h=2^{-4}$ obtained with \ascp and {\sc as-minres-bdf}, resp.,
as $\nu$ and $\epsilon$ vary. 

We observe that for $\beta=10$ the average number of inner iterations
becomes large when $\epsilon$ is small
and $\nu$ is large (top right corner) whereas for
$\beta=100$ the increase in iteration number is more evident 
in the opposite setting (bottom left corner).
Overall, the variation of the reported values is quite modest and
smallest values are located on the diagonal of the table (shaded cells), i.e. when $\nu=\epsilon^2$.
We recall that $\nu=\epsilon^2$ corresponds to $\gamma_1=\gamma_2=\frac 1 2$ 
in the block $L_1$ of the Schur approximation (\ref{eqn:shat_gen}),
so that Proposition \ref{prop:gamma12} holds (see Remark \ref{rem:nueps2}).
We also notice that the variation in the number of iterations is significantly
less pronounced for the indefinite preconditioner than for the
block diagonal preconditioner. In particular, for a fixed $\nu$, the
average number of iterations for \ascp varies very mildly. More significant
variations for fixed $\nu$ are visible for {\sc as-minres-bdf}, see Table \ref{MC_eps_nu_bdf}.
Moreover, we observe that the behavior of the proposed preconditioner
does not deteriorate for $\epsilon \rightarrow 0$ and, in particular, fully
satisfying results are obtained for $\epsilon =0$, i.e. in the solution of State Constrained problems.

We point out that similar digits were observed when using
a direct solver (not reported here) in place of {\sc hsl-mi20} within the preconditioners.
Therefore, the different performance as the parameters deviate from
$\nu=\epsilon^2$ is not due to the preconditioner inexactness, but
rather, to the different quality of the (exact) preconditioner itself.
{The only exception is given by the two runs marked with the symbol `*' in 
Table~\ref{MC_eps_nu_bdf}, for which a lower average number of {\sc minres}
iterations was observed when using a direct solver in place of {\sc hsl-mi20}.}

}

\end{example}

\subsubsection{The inexact active-set Newton method for CC problems} \label{exp_inex}

Performing the experiments on problems with CC constraints (\ref{pb_cc}),  we observed 
different trends in the nonlinear iteration progress varying the parameters 
$\nu$ and $\beta$, see e.g. the values of {\sc nli} in Table \ref{tabcpf_pb2}.
To clarify this issue,  we plot in Figure \ref{convhist} the convergence history of \ascp 
on {\tt CC-Pb1} with mesh size  $h=2^{-4}$ varying $\beta_1\in \{0,10,100,1000\}$ and 
setting $\nu = 10^{-2}$ in the left plot and $\nu = 10^{-6}$ in the right plot. 
\begin{figure}[htb]
\centering
\includegraphics[width=2.5in,height=2.1in]{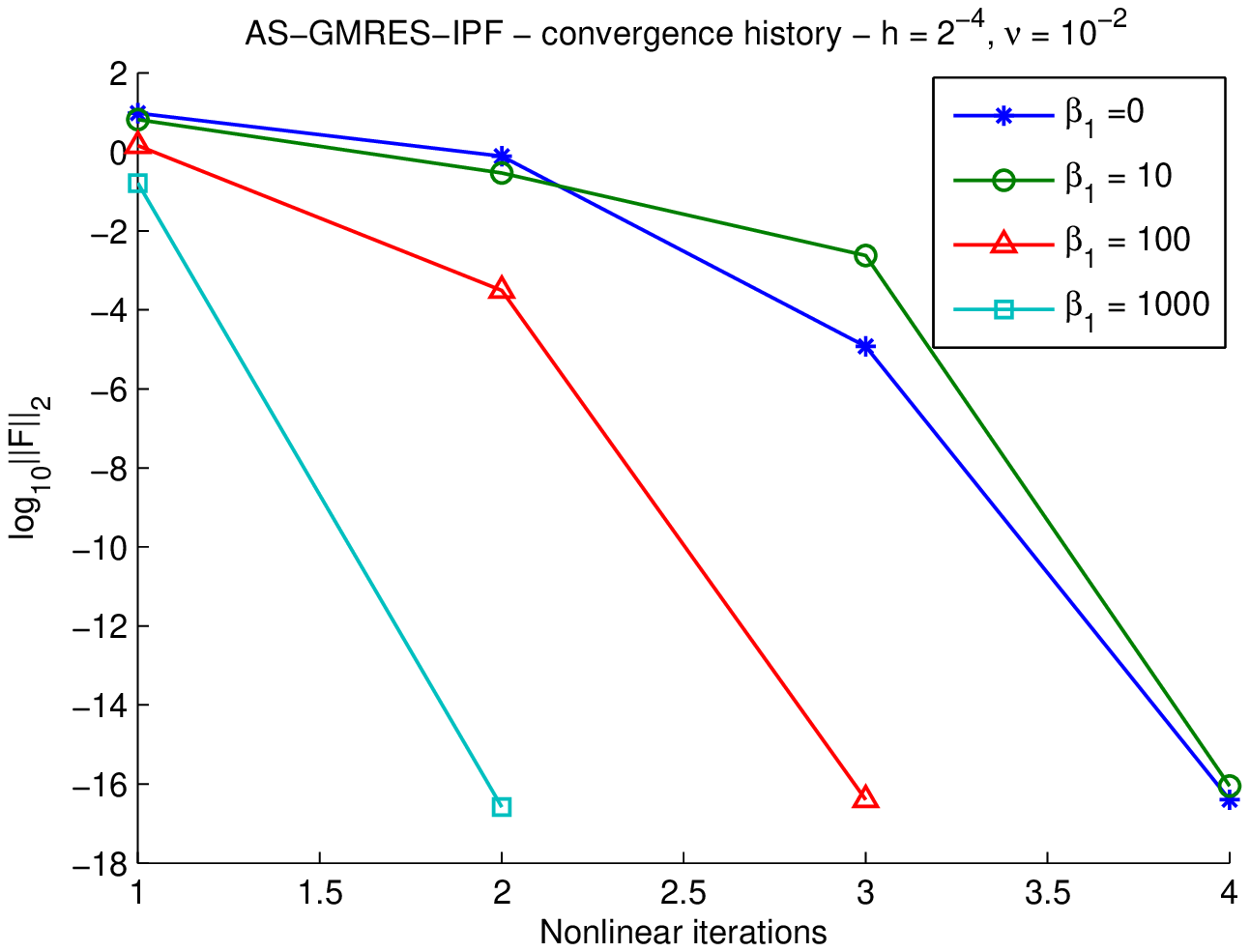}
\includegraphics[width=2.5in,height=2.1in]{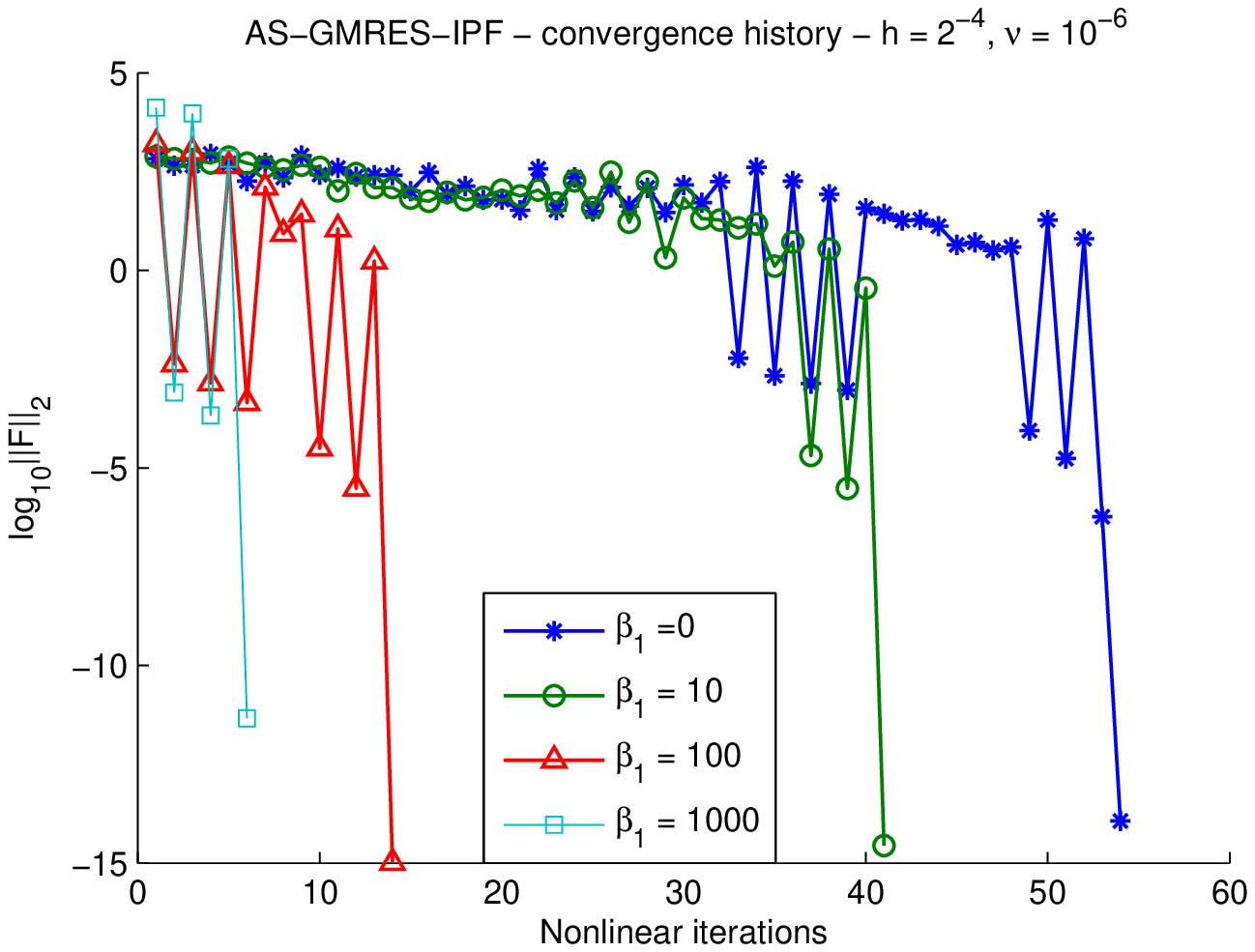}
\caption{Convergence history of  \ascp for the {\tt CC-pb1} with $h=2^{-4}$. Left: $\nu = 10^{-2}$.
Right: $\nu = 10^{-6}$. 
\label{convhist}}
\end{figure}

Looking at each plot we note that the number of nonlinear iterations  decreases as $\beta$ becomes larger;
moreover, comparing the two plots, we observe an increase
of Newton steps for a smaller $\nu$.
More interestingly, the right plot in Figure \ref{convhist} shows 
a long stagnation phase in the nonlinear process before reaching the
local area of fast Newton convergence. 
In this first phase, away from a solution, choosing an $\eta_k$ too small
(as in (\ref{ex_crit1})) can lead to oversolving the Newton equation (\ref{new_eq}): the corresponding
step may result in little or no progress toward a solution, while involving pointless
expense.

We therefore combined the active-set method with the 
inexact adaptive choice (\ref{in_crit}). 
We report in Table \ref{tab_inex} the results of \ascp using  the
adaptive value $\eta_k = \eta_k^I$ in (\ref{in_crit})
on problem {\tt CC-Pb1} with $h\in \{2^{-4},2^{-5}\}$,
$\beta_1 \in \{0,10,100,1000\}$ and $\nu\in\{10^{-2},10^{-4},10^{-6},10^{-8}\}$.

 \begin{table}[htb]
\begin{center}
\footnotesize
\begin{tabular}{|c|c|rr|rr|rr|rr|}
\hline
                    &            & \multicolumn{ 2}{c}{$\nu=10^{-2}$} & \multicolumn{ 2}{|c}{$\nu=10^{-4}$} & \multicolumn{ 2}{|c}{$\nu=10^{-6}$} & \multicolumn{ 2}{|c|}{$\nu=10^{-8}$} \\

   $\beta_1$ &        $p$ &    {\sc li}({\sc nli}) &       {\sc tcpu} &    {\sc li}({\sc nli}) &       {\sc tcpu} &    {\sc li}({\sc nli}) &       {\sc tcpu} &    {\sc li}({\sc nli}) &       {\sc tcpu} \\
\hline
     0 &   $4$ &    3.2(5) &       3.2 &    4.5(18) &      15.9 &  7.0(60) &   77.1 &   13.1(131) &  281.9 \\
       &   $5$ &    4.0(6) &       59.1 &    2.7(16)&     102.0 &  3.5(92) &  646.4&   -  &  - \\
			\hline
    10 &   $4$ &    3.5(4) &     3.2 &   3.2(14)  &11.1 &   5.2(60) & 65.9 &  112.7(101) &      233.9 \\
       &   $5$ &    4.0(5) &      39.6 &   2.9(15) & 100.7 &  3.0(53) &  377.7* &   -  &  - \\
\hline
   100 &   $4$ &       3.0(3) &    1.8 &   3.3(6) &   4.1 &   11.6(14) &   30.4 & 10.6(51) & 100.5 \\
       &   $5$ &    4.3(3) &      25.4 &   3.1(6) & 42.1 &  2.9(27) &  332.4* &   -  &  - \\

\hline
      1000 &   $4$ &      2.5(2) & 1.1 &  3.3(3) & 2.1 &  7.5(6) & 8.5 & 20.4(10) &   45.1 \\
       	   &   $5$ &    2.5(2) &    18.1 &  3.0(3) & 31.6 &  2.5(7) &  67.3 &   8.1(18)  &  924.1 \\

\hline
\end{tabular}  \end{center}
\caption{\ascp on {\tt CC-Pb1} for $h=2^{-p} \in\{2^{-4},2^{-5}\}$ and a variety of values for $\nu$ and $\beta$. The symbol `*' denotes runs where an  {\sc hsl-mi20} warning occurred.} \label{tab_inex}
\end{table}

Let us compare values in Table \ref{tab_inex} with the corresponding values in Table~\ref{tabcpf_pb2} (top table)
obtained with $\eta_k^E$. The average number of linear iterations is smaller in 
Table~\ref{tab_inex}  than in Table \ref{tabcpf_pb2} while the number of nonlinear iterations
is larger in 15 over 29 successful runs. Overall, the saving in number of inner iterations
of \ascp with $\eta_k^I$ makes it faster than \ascp with $\eta_k^E$ in all runs.
Two extra failures 
occur when $\eta_k^I$ is used in the limit case $\nu=10^{-8}$.

Summarizing, the inexact strategy is both cheaper and more effective in solving problem (\ref{NL}),
especially for  $\nu \in\{ 10^{-4}, 10^{-6}\}$ and $\beta_1\le10$, that is values for which the stagnation phase is longer.
Note that in particular, a less stringent inner accuracy allows a fast solution also
in the limit case $\beta_1 = 1000$.

\section{Conclusions}\label{conclusions}
We have proposed two classes of preconditioners (a positive definite one
and an indefinite one) for efficiently solving problem (\ref{pb_gen}) by means of
an active-set Newton method. Both acceleration strategies rely on a new effective
approximation to the Schur complement of the Jacobian matrix, for which
spectral estimates are provided.

A large set of numerical experiments shows the great potential of these
preconditioners for a large range of all problem parameters. As opposed
to the current literature, we cope with the indefiniteness of the problem by
appropriately choosing the structured preconditioner, 
and we include active set information explicitly in the preconditioning blocks to
exploit this information at later stages.
Therefore, the preconditioner adapts dynamically with the modification of the
active sets.
This procedure allowed us to devise a general and simple to implement
acceleration strategy, that can be employed either within {\sc minres} (in the
block diagonal form) or within {\sc gmres} (in the indefinite factorized form).
The latter formulation outperforms {\sc minres}  in all test cases, and
shows significantly lower sensitivity to the extreme values of the parameters.
In general, memory requirements of {\sc gmres} remain modest, as the number of iterations
stays quite small throughout the nonlinear process.
For the smallest values of $\nu$, however, the number of
{\sc gmres} iterations may make its memory requirements undesirably high. In this
case, a short-term recurrence such as the symmetric version of {\sc qmr}
could be considered as an alternative; see, e.g., \cite{PerugiaSimoncini00}
for a discussion and related numerical experiments.
{We also mention that a dimension reduction could be employed in
the original system (\ref{new_eq}). This strategy is discussed in
\cite{Simoncini.12} in the case when no bound constraints are imposed, and it could be
naturally generalized to our setting.}

Although some of the preconditioner blocks need to be recomputed at each
Newton iteration, this cost does not seem to penalize the overall
performance of the preconditioned solver. Numerical comparisons with state-of-the-art
methods available in the literature support these claims.

Finally, we mention that more general regularization terms could be
considered for the cost functionals, for instance, by enforcing 
sparsity constraints, see, e.g., \cite{Stadler09}. We aim to address this important
aspect in future research.

\section*{Acknowledgements} We would like to thank U. Langer, J. Pearson, M. Stoll, A. Wathen and W. Zulehner 
for helpful discussions on the topic of this paper. 
{\color{black} Finally, we aknowledge the insightful remarks of two anonymous referees.}

\bibliography{mybiblio}

\section*{Appendix}
In this appendix we collect some of the technical proofs (subscript $k$ is omitted).

\vskip 0.1in
{\it Proof of Lemma \ref{lemma:F}.} 
From $F+F^T\succeq0$ it also follows that $F+I$ is nonsingular.

i) We consider the eigenvalue problem $(F+I)^{-1}(F-I)(F-I)^T (F+I)^{-T} x = \theta x$ with $\theta\ge 0$, or, equivalently,
$(F-I)(F-I)^T y =\theta (F+I) (F+I)^T y$ with $y=(F+I)^{-T} x$. The largest eigenvalue coincides with
$\|(F+I)^{-1}(F-I)\|^2$.
 We have
$(F-I)(F-I)^T = FF^T + I -F - F^T$ and $(F+I) (F+I)^T = FF^T + I +F + F^T$. Substituting and
rearranging terms gives
$$
(1-\theta) (FF^T+I) y = (\theta +1)(F+F^T) y .
$$
We multiply from the left by $y^T$. Since $FF^T+I\succ0$, $F+F^T\succeq 0$ and $\theta+1>0$, it must be that
$1-\theta \ge0$, that is $\theta \le 1$.

ii) We proceed in a similar way. Let us now consider
 $(F+I)^{-1}(F+F^T)(F+I)^{-T} x = \theta x$, with $\theta >0$, which is equivalent to
$(F+F^T) y = \theta (F+I)(F+I)^T y$, with $y=(F+I)^{-T} x$. Therefore,
$(1-\theta) (F+F^T) y = \theta(FF^T+I) y$.  We premultiply by $y^T$ and rearrange to obtain
$$
\frac{1 - \theta}{\theta} = \frac{y^T (FF^T+I) y}{y^T (F + F^T) y}.
$$
From the relation $(F - I)(F - I)^T \succeq 0 $ it follows that $\displaystyle \frac{y^T (FF^T+I) y}{y^T (F + F^T) y} \geq 1$. Thus, $\displaystyle \frac{1 - \theta}{\theta} \geq 1$ which implies $\theta \leq \displaystyle \frac{1}{2}$.
\endproof

\vskip 0.2in

{\it Proof of Proposition \ref{prop:upperestimate}.} Let $F=\sqrt{\nu} M^{-\frac 1 2} L M^{-\frac 1 2}$.

i) For $\alpha_u=1$ and $\alpha_y=0$ we have $\gamma_1=0$ and $\gamma_2=1$,
so that
$M^{-\frac 1 2} {\mathbb S} M^{-\frac 1 2} = FF^T + (I-\Pi)$,
and
$$
M^{-\frac 1 2} \widehat {\mathbb S} M^{-\frac 1 2} =
(F+(I-\Pi)) (F+(I-\Pi))^T ,
$$
where we used the fact that $(I-\Pi)^{\frac 1 2} = (I-\Pi)$.
From $M^{-\frac 1 2} {\mathbb S} M^{-\frac 1 2}  x = \lambda M^{-\frac 1 2} \widehat {\mathbb S} M^{-\frac 1 2} x$
we obtain for $y=M^{\frac 1 2}x$
\begin{eqnarray}\label{eqn:Feig}
(F+(I-\Pi))^{-1} (FF^T + (I-\Pi)) (F+(I-\Pi))^{-T} y = \lambda y .
\end{eqnarray}
Since $F$ is nonsingular, we have
\begin{eqnarray*}
&&(F+(I-\Pi))^{-1} (FF^T + (I-\Pi)) (F+(I-\Pi))^{-T}  \\
&\qquad & =
(F+(I-\Pi))^{-1}F (I + F^{-1}(I-\Pi)(I-\Pi)F^{-T} ) F^T (F+(I-\Pi))^{-T}  \\
&\qquad & = (I+F^{-1}(I-\Pi))^{-1} (I + F^{-1}(I-\Pi)(I-\Pi)F^{-T} )  (I+F^{-1}(I-\Pi))^{-T}  \\
&\qquad& =: (I+Z)^{-1} (I+ZZ^T)(I+Z)^{-T} ,
\end{eqnarray*}
with $Z=F^{-1}(I-\Pi)$. Therefore, from (\ref{eqn:Feig}) it follows
\begin{eqnarray}
\lambda  &\le& 
 \|(I+Z)^{-1} (I+ZZ^T)(I+Z)^{-T}\|  
\le  \|(I+Z)^{-1}\|^2 + \|(I+Z)^{-1}Z\|^2  \nonumber\\
&=& \|(I+Z)^{-1}\|^2 + \|I- (I+Z)^{-1}\|^2 \nonumber \\
&\le& \|(I+Z)^{-1}\|^2 + (1 + \|(I+Z)^{-1}\|)^2.\label{eqn:lambda}
\end{eqnarray}
We then recall that
$Z = F^{-1}(I-\Pi) = 
 \frac 1 {\sqrt{\nu}} M^{\frac 1 2} L^{-1} M^{\frac 1 2} (I-\Pi)$,
so that
\begin{eqnarray*}
\|(I+Z)^{-1}\| &=& \|(I+ \frac {1} {\sqrt{\nu}} M^{\frac 1 2} L^{-1} M^{\frac 1 2} (I-\Pi))^{-1} \| 
\\
& = & 
\|M^{\frac 1 2} \left (\sqrt{\nu} L  +M (I-\Pi)\right )^{-1} \sqrt{\nu} L M^{-\frac 1 2} \|  .
\end{eqnarray*}

To analyze the behavior for $\nu\to 0$, let us suppose that $L + L^T \succ 0$, and
write $Z= \frac 1 {\sqrt{\nu}} \widetilde F^{-1}(I-\Pi)$; 
without loss of generality also assume that $I-\Pi = {\rm blkdiag}(I_\ell,0)$.
The eigendecomposition of $\widetilde F^{-1}(I-\Pi)$ is given by\footnote{In the 
unlikely case of a Jordan decomposition,
the proof proceeds with the maximum over norms of 
Jordan blocks inverses, which leads to the same final result.}
$\widetilde F^{-1}(I-\Pi) = X \Lambda X^{-1}$ where $\Lambda={\rm diag}(\lambda_i)$
and $\lambda_i \in {\rm spec}((\widetilde F^{-1})_{11}) \cup \{0\}$. Here $(\widetilde F^{-1})_{11}$ is
the top left $\ell\times \ell$ block of $\widetilde F^{-1}$.
Note that 
all eigenvalues of $(\widetilde F^{-1})_{11}$
have strictly positive real part, thanks to the condition $L+L^T \succ 0$. 
Therefore 
{\small
\begin{eqnarray*}
\|(I+Z)^{-1}\| & = &
\|X (I + \frac 1 {\sqrt{\nu}} \Lambda)^{-1} X^{-1} \| 
\le
{\rm cond}(X) 
\max \left\{ \frac 1 {\displaystyle\min_{\lambda\in {\rm spec}((\widetilde F^{-1})_{11})} |1+\lambda/\sqrt{\nu}|}, 1\right\} .
\end{eqnarray*}
}
We thus have
$$
\max \left\{ \frac 1 {\displaystyle \min_{\lambda\in {\rm spec}((\widetilde F^{-1})_{11})} |1+\lambda/\sqrt{\nu}|}, 1\right\} \to 1 
\qquad {\rm for } \qquad \nu \to 0 ,
$$
so that $\|(I+Z)^{-1}\| \le \eta\, {\rm cond}(X)$ with $\eta\to 1$ for $\nu \to 0$.

ii) For $\alpha_u=0$ and $\alpha_y=1$ we have $\gamma_1=1$ and $\gamma_2=0$,
so that
$M^{-\frac 1 2} {\mathbb S} M^{-\frac 1 2} = F(I-\Pi)F^T + I$,
and
$$
M^{-\frac 1 2} \widehat {\mathbb S} M^{-\frac 1 2} =
(F(I-\Pi)+I) (F(I-\Pi)+I)^T .
$$
As before, setting this time $Z= F(I-\Pi)$ we obtain
the bounds (\ref{eqn:lambda}) for $\lambda$ 
with $\|(I+Z)^{-1}\| = \|(I+\sqrt{\nu}M^{-\frac 1 2}LM^{-\frac 1 2}(I-\Pi))^{-1}\|$.
Finally, it is apparent from the above expression that $\|(I+Z)^{-1}\| \to 1$ as $\nu \to 0$.
\endproof

\end{document}